\documentclass[11pt,a4paper,leqno]{article}
\usepackage{amsmath}
\usepackage{amssymb}
\usepackage{color}
\advance\topmargin by -2.2cm
\newcommand{\bbr}{I\!\!R}
\newcommand{\bbt}{T\!\!\!T}
\newcommand{\bbn}{I\!\!N}
\newcommand{\bbz}{Z\!\!\!Z}

\newcommand{\cala}{{\cal A}}
\newcommand{\calb}{{\cal B}}
\newcommand{\calc}{{\cal C}}

\newcommand{\cale}{{\cal E}}
\newcommand{\calf}{{\cal F}}
\newcommand{\calg}{{\cal G}}

\newcommand{\call}{{\cal L}}

\newcommand{\calp}{{\cal P}}

\newcommand{\cals}{{\cal S}}

\newcommand{\barr}{\begin{array}}
\newcommand{\earr}{\end{array}}
\newcommand{\beqq}{\begin{equation}}
\newcommand{\eeqq}{\end{equation}}
\newcommand{\beao}{\begin{eqnarray*}}
\newcommand{\eeao}{\end{eqnarray*}\noindent}
\newcommand{\beam}{\begin{eqnarray}}
\newcommand{\eeam}{\end{eqnarray}\noindent}

\newcommand{\la}{\lambda}
\newcommand{\La}{\Lambda}

\newcommand{\si}{\sigma}
\newcommand{\al}{\alpha}

\newcommand{\om}{\omega}
\newcommand{\Om}{\Omega}

\newcommand{\vph}{\varphi}

\newcommand{\vep}{\varepsilon}

\newtheorem{theo}{Theorem}
\newtheorem{prop}{\indent Proposition}

\newtheorem{rem}{\indent Remark}

\newtheorem{lem}{\indent Lemma}
\newtheorem{defin}{\indent Definition}
\newtheorem{cor}{\indent Corollary}

\newtheorem{ex}{\indent Example}
\newtheorem{ass}{\indent Assumption}

\newenvironment{proof}{\noindent {\bf Proof }}
{\hfill $\bullet$ \vspace{0.25cm}}

\newcommand{\wh}{\widehat}
\newcommand{\wt}{\widetilde}

\newcommand{\ov}{\overline}

\newcommand{\Lra}{\Longrightarrow}

\newcommand{\lra}{\longrightarrow}

\newcommand{\tto}{t\to\infty}

\newcommand{\nto}{n\to\infty}

\author{R. H\"opfner$^*$ \and E.~L\"ocherbach \and M. Thieullen\thanks{This work has been supported by the Agence Nationale de la Recherche through the project MANDy, Mathematical Analysis of Neuronal Dynamics, ANR-09-BLAN-0008-01. e-mail addresses: {\tt hoepfner@mathematik.uni-mainz.de, eva.loecherbach@u-cergy.fr, michele.thieullen@upmc.fr} }\\
{\it Johannes Gutenberg-Universit\"at Mainz, Universit\'e de Cergy-Pontoise} \\
 {\it and Universit\'e Pierre et Marie Curie.}}
\begin{document}
%%%%%%%%%%%%%%%%%%%%

\title{Ergodicity and limit theorems for degenerate diffusions \\ with time periodic drift. \\ Application to a stochastic Hodgkin-Huxley model.}

\maketitle

\begin{abstract}
We formulate simple criteria for positive Harris recurrence of strongly degenerate stochastic differential equations with smooth coefficients on a state space with certain boundary conditions.  The drift depends on time and space and is periodic in the time argument. There is no time dependence in the diffusion coefficient. Control systems play a key role, and we prove a new localized version of the support theorem. Beyond existence of some Lyapunov function, we only need one attainable inner point of full weak Hoermander dimension.  
 
Our motivation comes from a stochastic Hodgkin-Huxley model for a spiking neuron including its dendritic input. This input carries some deterministic periodic signal coded in its drift coefficient and is the only source of noise for the whole system. We have a 5d SDE driven by 1d Brownian motion. As an application of the general results above, we can prove positive Harris recurrence. Here analyticity of the coefficients and 
Nummelin splitting allow to formulate a Glivenko-Cantelli type theorem for the interspike intervals.  
\end{abstract}

{\small 
 }%endesmall

%%%%%%%%%%%%%%%%%%%%%%%%%%%%%%%
\section{Introduction}
%%%%%%%%%%%%%%%%%%%%%%%%%%%%%%%

Consider a $d$-dimensional diffusion driven by $m$-dimensional Brownian motion 
\beqq\label{sde}
dX_t \;=\; b(t,X_t)\, dt\;+\; \si(X_t)\, d W_t   \quad,\quad   t\ge 0   
\eeqq 
where $m\le d $, with smooth (and sometimes even analytic) coefficients 
\beqq\label{coefficients}
b(t,x) \;=\; \left( \begin{array}{l}  
b^1(t,x) \\ \quad\vdots \\ b^d(t,x) 
\end{array}\right) \quad,\quad 
\si(x) \;=\; \left( \begin{array}{lll}  
\si^{1,1}(x) & \quad\ldots\quad & \si^{1,m}(x) \\
\quad\vdots &   & \quad\vdots \\
\si^{d,1}(x) & \quad\ldots\quad & \si^{d,m}(x) 
\end{array}\right) \;. \\[2mm]
\eeqq
We require that the coefficients be such that for every starting point a unique strong solution exists and has infinite life time; the state space for process (\ref{sde}) will be some Borel subset $E$ of $\bbr^d$ with Borel-$\si$-field $\cale$, see Section \ref{sec:setting},  
satisfying certain boundary conditions.

We are interested in Harris properties of the process $X$ (which is non-homogeneous in time) under the assumption that the drift is periodic in the time argument, with main focus on case $m<d$ where the SDE (\ref{sde}) is degenerate. 
Our criteria will be in terms of control systems and the support theorem --assuming principally that there exists one inner point of the state space which is of full weak Hoermander dimension and attainable in a sense of deterministic control-- and in terms of some Lyapunov function. 
We use \cite{HH-TD} for existence of smooth transition densities locally at the attainable point. 

Our assumptions on the state space and on the coefficients of the SDE's are such that 
results which we prove for general degenerate SDE's can be applied to stochastic Hodgkin-Huxley models where dendritic input is the only source of noise, and where a determinisitic $T$-periodic signal is encoded in the semigroup of the process. For such models we can establish positive Harris recurrence.

Our main results in the general setting are positive Harris recurrence of strongly degenerate SDE's  (Theorem \ref{theo:1}) and a new localized version of the support theorem (Theorem \ref{theo:4bis}) adapted to the specific structure of our state space. For stochastic Hodgkin-Huxley models, we obtain the Harris property (Theorem \ref{theo:2})  from explicit construction of a Lyapunov function and of a control steering trajectories towards the attainable point. We do this for Cox-Ingersoll-Ross type input and for Ornstein-Uhlenbeck type input. In the latter case, we improve on \cite{OU-HH} since we obtain a unique Harris set. 
We need the Harris property to give a sound mathematical meaning to "typical spiking behaviour of a neuron" in terms of long time behaviour, and prove convergence of empirical distribution functions for the interspike times (Glivenko-Cantelli type Theorem \ref{theo:4}) by combining analyticity of the coefficients with the Harris property.  

The plan of the paper is as follows: the setting and the assumptions under which we are working are explained in Section \ref{sec:setting}. We give the main result on Harris recurrence of degenerate SDE's in Section \ref{sec:results}, its application to stochastic Hodgkin-Huxley systems in Sections \ref{sec:HH} and \ref{sec:limit theorems}. The proofs for degenerate SDE's are in Sections \ref{sec:harris} and \ref{sec:analyticharris}, the proofs for stochastic Hodgkin-Huxley models in Section \ref{sec:5}. Sections \ref{sec:control1}, \ref{sec:hoer}, \ref{sec:control} and \ref{sec:app} can be read independently: here we deal with weak Hoermander condition and we state the localized version of the control theorem, in a setting where the classical techniques apply only locally.

%%%%%%%%%%%%%%%%%%%%%%%%%%
\section{Outline of results} %Assumptions and main results

%%%%%%%%%%%%%%%%%%%%%%%

%%%%%%%%%%%%%%%%%%%%%%%
\subsection{The setting and the assumptions} %Main assumptions
\label{sec:setting}
%%%%%%%%%%%%%%%%%%%%%%%%%

We start with an example in order to motivate our main assumptions.

\begin{ex}\label{ex:1}
Consider the two dimensional process $(X_t)_{t\ge 0}$ defined by 
\beqq\label{toyexample}
X=\left(\begin{array}{l} \xi \\ \psi \end{array}\right) \quad,\quad 
\left\{\begin{array}{lll} 
d\xi_t &= & -c\, \sin^2(2\pi t)\, \xi_t\, dt + dB_t \\ 
d\psi_t &= & (1-\psi_t)\, dt + \psi_t\, dB_t
\end{array}\right\} \;,
\eeqq
where $ (B_t)_{t \geq 0 } $ is one dimensional Brownian motion.  
We have an explicit representation of the solution:  
$\,\xi_t = \xi_0\, e^{ - c \int_0^t \sin^2(2\pi v) dv} + \int_0^t e^{ - c \int_r^t \sin^2(2\pi v) dv} dB_r\;$ for  $0<t<\infty$, and  
$\,\psi_t = ( \psi_0 + \int_0^t \Phi_r^{-1} dr )\,\Phi_t\,$ where  $\,\Phi_t := e^{-t} e^{B_t-\frac12 t}\,$ takes values in $(0,\infty)$. Notice that $\psi_0\ge 0$ implies $\psi_t>0$ for all $0<t<\infty$. Choosing $E:=\bbr{\times}[0,\infty)$ as state space for the process (\ref{toyexample}), $\,\partial(E)\cap E$ is an entrance boundary of $X.$  
\end{ex}

The structure of the state space described in Example \ref{ex:1} is typical for what we will consider in this paper. Our first standing assumption is

\begin{ass}\label{ass:1}
a) For some strictly increasing sequence $(G_m)_m$ of bounded convex open sets in $\bbr^d$ and  compacts $C_m := {\rm cl}(G_m)$, we have  
$\,E = \mathop{\bigcup}\limits_m C_m\,$ and $\,\cale:=\calb(E)\,$.\\
b) $\partial(E)\cap E$ is an entrance boundary for the process $X.$ \\
b') From positions $x\in C_m\setminus G_{m+1}$, almost surely, the process $X$ immediately enters $G_{m+1}$.
\\
c) Defining stopping times $T_m:= \inf\{ t > 0 : X_t\notin C_m\}$ for the process, we have $T_m\uparrow \infty$ as $m\to\infty$ almost surely, for every choice of a starting point $x\in E$. \\
d) The components of coefficients (\ref{coefficients}) for equation (\ref{sde}) 
$$ (t, x ) \to b^i  (t, x) , \; x \to \sigma^{i, j } (x) ,  i = 1, \ldots , d, j = 1 , \ldots , m$$
are $ C^\infty -$ functions on $ \bbr_+ \times U,$ for some open set $U \subset \bbr^d$ which contains $E$.
\end{ass}

The above assumption combines properties of the process (such as non-explosion or behavior at the boundary) with topological properties of $E$.  
It might hold for suitable processes (\ref{sde}) with $E$ open in $\bbr^d$ --thus $\partial(E)\cap E=\emptyset$ in b)-- for a compact exhaustion $(C_m)_m$ of $E$ in a), i.e.\ $C_m\subset G_{m+1}$ for all $m$. 
In case of Example 1, it holds with $G_m := ( - m , m ) \times ( 0, m ) $ and $E=\bbr\times[0,\infty)$  since the process (\ref{toyexample}) starting in $[-m,m]{\times}\{0\}=C_m\setminus G_{m+1}$ immediately enters $G_{m+1}$. 
Assumption \ref{ass:1} will be what we need for the stochastic Hodgkin-Huxley system of Section \ref{sec:HH}. It implies  
in particular that $(E,\cale)$ is a Polish space. The compacts $C_m={\rm cl}(G_m)$ are needed for various localization purposes.

We write $P_{s,t}(x,dy)$, $0\le s < t<\infty$, for the semigroup of transition probabilities on $(E,\cale)$. For $x\in E$ and $s\ge 0$, we write $Q_{(s,x)}$ for the law of the process (\ref{sde}) starting at time $s$ in $x$, a probability on the path space 
$\Om:=C([s,\infty),E)$ 
with the topology of locally uniform convergence; then the Borel $\si$-field $\cala$ is also generated by the coordinate projections, and $(\Om,\cala$) is again a Polish space. In case $s=0$ we write for short $Q_x:=Q_{(0,x)}$, $x\in E$. Finally, we also denote by $Q_x^{t_0}$ the law of the solution $ (X_t)_{0 \le t \le t_0} $ of  \eqref{sde}, for a finite time horizon $[0, t_0],$ starting from $X_0 = x ,$ on $C([0,t_0],E).$ On $(\Om,\cala)$, equipped with the canonical filtration $\mathbb{G}=(\calg_t)_{t\ge 0}$ and with shift operators $(\theta_t)_{t\ge 0}$, $X$ is simply the canonical process.

Our second standing assumption is time-periodicity of the drift together with existence of a Lyapunov function which constitutes a first step towards Harris property.

\begin{ass}\label{ass:2}
a) We take the drift $T$-periodic in the time variable: 
$$
b(t,x) \;=\; b( i_T(t) , x ) \quad,\quad\mbox{$i_T(t) := $ $t$ modulo $T$} \;. 
$$
b) We have a {\em Lyapunov function} $V:E\to [1,\infty)$, in the following sense: $V$ is $\cale$-measurable; there is a compact $K$ contained in $E$ (i.e., $K\subset E$ and $K\neq E$) and some $\vep>0$ such that 
$$
%K\subset G_m \;,\quad 
\mbox{$P_{0,T}V$ is bounded on $K$} \;,\quad 
P_{0,T}V \;\le\; V \;-\; \vep \quad\mbox{on $\,E\setminus K$} \;. 
$$ 
\end{ass}
\vskip0.5cm

By Assumption \ref{ass:2} a), the semigroup of the process (\ref{sde}) is $T$-periodic in time which means that
$$
P_{s,t}(x,dy) \;=\; P_{s+kT,t+kT}(x,dy) \;\;,\;\; k\in\bbn_0.   
$$
This implies that the {\em $T$-skeleton chain} $(X_{kT})_{k\in\bbn_0}$ is a time-homogeneous Markov chain. By Assumption \ref{ass:2} b),  
$\,(V(X_{kT}))_{k\in\bbn_0}\,$  
evolves as a nonnegative supermartingale as long as it stays outside $K$. As a consequence, the skeleton chain has to visit the compact $K$ infinitely often, almost surely, for arbitrary choice of a starting point in $E$. 
We ask the following question: which additional condition grants that the $T$-skeleton chain is recurrent in the sense of Harris? 

The $T$-periodicity of the semigroup has another important consequence: as in the proof of Proposition 5 in \cite{OU-HH}, we will have equivalence of the Harris property for the $d$-dim skeleton chain $(X_{kT})_{k\in\bbn_0}$ and of the Harris property for the $(1{+}d)$-dim continuous time process 
$$ 
\ov{X} = (\ov{X}_t)_{t\ge 0} \quad\mbox{defined by}\quad  \ov{X}_t := (i_T(t),X_t)   
$$ 
taking values in $\,\ov E := \bbt\times E\,$  where $\bbt := [0,T]$ is the torus, identifying $t$ with $i_T(t)$. $\,\ov E$ equipped with its Borel-$\si$-field $\ov \cale$ is a Polish space.  Time being included as a 'zero-component' into the continuous-time process, $\,\ov X$ is homogeneous in time. For the Harris property in discrete time see Harris \cite{Har-56} and Revuz \cite{Rev-84} p.\ 92, for the Harris property in continuous time see Az\'ema, Duflo and Revuz \cite{ADR-69}. The (unique up to constant multiples) invariant measures $\mu$ of $(X_{kT})_{k\in\bbn_0}$  and $\ov\mu$ of $\ov X$ will be related by  
\beqq\label{1-1-correspondence}
\ov{\mu}{(dt,dx)} \;=\; \frac1T \int_0^T ds\; [\epsilon_{ \{ s\} }\otimes \mu P_{0,s}]({ds},dx)  \;. 
\eeqq
By the form of the Lyapunov condition in Assumption \ref{ass:2}~b), recurrence will necessarily be positive recurrence: we thus will have $\mu(E)<\infty$, and up to choice of a norming factor $\mu$ will be a probability measure on $(E,\cale)$,  $E\subset \bbr^d$ (cf. Meyn and Tweedie \cite{MeyTwe-92}, Theorem 4.3).

Next we introduce a notion of `attainable points' for the process $X$. Note that this notion is entirely deterministic.  
Stratonovich drift $\wt b(\cdot,\cdot)$ is specified in (\ref{strato}).

\begin{defin}\label{def:1}
A point $x^*$ in $E$ is called {\em attainable in a sense of deterministic control} if it belongs to ${\rm int}(E)$ and if the following holds:  \\
for arbitrary $x\in E$, we can find some $\,\dot{\tt h}\,$ in $L^2_{\rm loc}$ (the class of $m$-dimensional measurable functions with components $\dot{\tt h}^{\ell}$ satisfying $\,\int_0^t \,[\dot{\tt h}^{\ell}(s)]^2\, ds < \infty\,$ for all $t<\infty$, $\ell=1,\ldots, m$) depending on $x$ and $x^*$ which drives the deterministic control system with Stratonovich drift $\wt b(\cdot,\cdot)$
$$
 d \vph (t) = \wt b ( t, \vph (t) ) dt + \si( \vph (t) ) \dot{\tt h}(t) dt, 
$$
from starting point $\vph(0)=x$ towards the limit  $\,x^* = \lim\limits_{t\to\infty}\vph (t)\,$,  
under the constraint $  \vph (t) \in E $ for all $ t \ge 0.$  
In this case we set $\vph:=\vph^{({\tt h}, x, x^*)}$.
\end{defin}

An illustration in the framework of Example \ref{ex:1} is given below.  We introduce our third key assumption, anticipating on Section \ref{sec:hoer} where we define `full weak Hoermander dimension'  (Definition \ref{def:hoerm2} in Section \ref{sec:hoer} does require some terminology, to be prepared there and to be read independently).

\begin{ass}\label{ass:3}
There is a point $\,x^*\in {\rm int}(E)\,$ with the following two properties: $\,x^*$ is of {\em full weak Hoer\-mander dimension} (Definition \ref{def:hoerm2} in Section \ref{sec:hoer} below), and  $\,x^*$ is {\em attainable in a sense of deterministic control}.  
\end{ass}

Notice that in the above assumption, it is sufficient to verify the weak Hoermander condition {\em at only one point} $ x^* \in {\rm int}(E).$  
This is easy for the process (\ref{toyexample}) of Example \ref{ex:1} (cf.\ end of this subsection and end of subsection \ref{sec:hoer}).  
For the stochastic Hodgkin-Huxley model of Section \ref{sec:HH}, we will be able to verify Assumption \ref{ass:3}.

At some point we will use a stronger version of Assumption \ref{ass:1} d) and suppose that the coefficients of \eqref{sde} are real analytic functions. This will allow to establish that weak Hoermander dimension remains constant along control paths as considered in Definition \ref{def:1} (which would not %necessarily 
be true for $\calc^\infty$ coefficients as considered so far). %As a consequence in In this case all points $x\in E$ will be of full weak Hoermander dimension (see  Section \ref{sec:control} below, respectively Lemma \ref{lem:1} and Theorem \ref{theo:6}). 
We introduce our last assumption, specific to the analytic case.

\begin{ass}\label{ass:5}
The components of coefficients (\ref{coefficients}) for equation (\ref{sde}) 
$$
(t,x)\to b^i(t,x) \quad,\quad x\to \si^{i,j}(x) \quad,\quad i=1,\ldots,d \;,\; j=1,\ldots,m
$$
are {\em real analytic} functions on $\bbt{\times}U$, for some open set $U\subset\bbr^d$ which contains $E$. 
\end{ass}

{\bf Example \ref{ex:1} continued}\quad  
{\it Put $d=2$, $m=1$, and consider the 2-dim process (\ref{toyexample}) of Example \ref{ex:1} driven by 1-dim Brownian motion, with state space $E=\bbr\times[0,\infty)$. Here the point $x^*:= (0,\frac23)$ is attainable in a sense of deterministic control. To check this, write $t\to \wt\xi(t)$ for the first and  $t\to \wt\psi(t)$ for the second component of a deterministic control system $t\to\vph^{({\tt h},\cdot,\cdot)}(t)$ with $\dot h \in L^2_{\rm loc}$. Write $x=({\xi \atop \psi})$ for points in $E$, $\,\xi\in\bbr$, $\psi\in [0,\infty)$. Put $\,S(t):=\int_0^t c \sin^2(2\pi v) dv\,$ for $t\ge 0$. The process (\ref{toyexample}) has Stratonovich drift $\,\wt b(t,x) = ({ -c \sin^2(2\pi t)\, \xi \atop 1 - \frac32 \psi })$, from (\ref{strato}). 
We have to determine $\dot h \in L^2_{\rm loc}$ such that simultaneously for all choices of a starting value $\wt\xi(0)\in\bbr$,  
$$
\wt\xi(t) := \wt\xi(0) e^{-S(t)} + \int_0^t \dot h(v) e^{-(S(t)-S(v))} dv 
\quad\mbox{solution to}\quad 
\frac{d\wt\xi(t)}{dt} = - c\, \sin^2(2\pi t)\, \wt\xi(t)   +  \dot h(t)  
$$
converges to the limit value $0$, and such that for all choices of a starting value $\wt\psi(0) \in [0,\infty)$ the solution to 
$$
\frac{d\wt\psi(t)}{dt} = 1 - \frac32 \wt\psi(t) +  \wt\psi(t) \dot h(t)   
$$
converges to $\frac23$ as $t\to\infty$. The first requirement is satisfied whenever $\,\dot h\,$ is a smooth function on $[0,\infty)$ which decreases to $0$ as $t\to\infty$. If we choose in particular $\dot h(0):=1$ and $\dot h(t)\equiv 0$ for $t\ge t_0$, then 
$T(t):= \int_0^t ( \frac32 - \dot h(v) ) dv$ is after time $t_0$ a linear function, so as desired,
$$
\wt\psi(t) := \wt\psi(0) e^{-T(t)} + \int_0^t e^{-(T(t)-T(v))} dv  \;\;\lra\;\;  \int_0^\infty e^{- \frac32 y} dy =\frac23 \quad\mbox{as $\tto$.}
$$
}

%%%%%%%%%%%%%%%%%%%%%%%%%%%%%
\subsection{Main results}\label{sec:results}
%%%%%%%%%%%%%%%%%%%%%%%%%%%%%

Now we can state the main {results of our paper. They strengthen the results obtained in \cite{OU-HH}.

\begin{theo}\label{theo:1}
Under Assumptions \ref{ass:1}--\ref{ass:3} the following holds: 

a) The skeleton chain $(X_{kT})_{k\in\bbn_0}$ is positive Harris recurrent.  

b) The continuous-time process  $\ov X = \left( (i_T(t),X_t) \right)_{t\ge 0}$ is a positive Harris recurrent process. 
\end{theo}

As a consequence of Theorem \ref{theo:1} and positive recurrence, see e.g.\ \cite{ADR-69}, \cite{Rev-84}, we obtain 
strong laws of large numbers. Note that by H\"opfner and Kutoyants \cite{H-K}, Section 2, there is a third equivalence to assertions a) or b) of Theorem \ref{theo:1}, i.e.\ Harris recurrence of the $T$-segment chain $\left( (X_{kT+s})_{0\le s\le T} \right)_{k\in\bbn_0}$, taking values in $C([0,T], E)$, and with invariant measure as specified there. Actually, part b) of the following Corollary \ref{cor:1}  
is a strong law of large numbers for additive functionals of the $T$-segment chain, applying Theorem 2.1 of \cite{H-K}.

\begin{cor}\label{cor:1} 
a) Grant Assumptions \ref{ass:1}-- \ref{ass:3}  and consider functions $G:E\to\bbr$ which belong to $L^1(\mu)$, and functions $F:\ov E\to\bbr$ which belong to $L^1(\ov{\mu})$. Then we have 
$$
\frac1n \sum_{k=1}^n G\left( X_{kT} \right)  \;\;\lra\;\;  \int_{\bbr^d} \mu(dy)\, G(y) \quad\mbox{$Q_x$-almost surely as $\nto$}
$$
$$
\frac1t \int_0^t F\left( i_T(s) , X_s \right) ds \;\;\lra\;\; \frac1T \int_0^T ds \int_{\bbr^d} [\mu P_{0,s}](dy)\, F(s,y) 
\quad\mbox{$Q_x$-almost surely as $t\to\infty$}\\[2mm]
$$
for arbitrary choice of a starting point $x\in E$. 

b) The second assertion in a) can be extended to $ \sigma-$finite measures $\La(ds)$ on $(\bbr,\calb(\bbr))$ which are $T$-periodic i.e.\ $ \Lambda ( B ) = \Lambda ( B + k T ) $ for any $ B \in \calb ( \bbr) $ and $ k \in \bbz$, as follows: 
$$
\frac1t \int_0^t F\left( i_T(s) , X_s \right) \Lambda (ds) \;\;\lra\;\; \frac1T \int_0^T \Lambda (ds) \int_{\bbr^d} [\mu P_{0,s}](dy)\, F(s,y) 
\quad\mbox{$Q_x$-almost surely as $t\to\infty$}
$$
provided the mapping $[0,T] \ni s \to \mu P_{0,s}|F(s,\cdot)|$ belongs to $L^1(\La)$. 
\end{cor}

In the analytic case, we obtain additionally to Theorem \ref{theo:1} and Corollary \ref{cor:1}:

\begin{prop}\label{cor:1bis}
Under Assumptions \ref{ass:1}--\ref{ass:5} the following statements hold true.\\ 
a) The weak Hoermander condition holds on the full state space~$E$. \\
b) The process (\ref{sde}) is a strong Feller process. \\ 
c) The invariant probability $\mu(dy)$ on $E$ and the projection 
$ 
\frac1T\int_0^T ds\, [\mu P_{0,s}](dy)
$   
of $\,\ov\mu(ds,dy)$ on its second component $y\in E$ admit Lebesgue densities.  
\end{prop}

Theorem \ref{theo:1} will be proved in Section \ref{sec:harris} below, Proposition \ref{cor:1bis} in Section \ref{sec:analyticharris}.

%%%%%%%%%%%%%%%%%%%%%%%%%%%%%%%
\subsection{Application: stochastic Hodgkin-Huxley systems where dendritic input is the only source of randomness}\label{sec:HH} 
%%%%%%%%%%%%%%%%%%%%%%%%%%%%%%%

The classical deterministic 4d Hodgkin-Huxley model for a spiking neuron (Hodgkin and Huxley \cite{HodHux-51}) consists of four variables: the voltage $v$ taking values in $\bbr$ which can be measured by introducing an electrode into the soma of the neuron, and three gating variables $\,n$, $m$, $h\,$ taking values in $[0,1]$ which represent the probabilities that `guardians' of certain types open certain types of ion channels. A deterministic function represents input as a function of time (e.g., a sequence of pulses). A modern introduction to the biological background can be found in the book by Izhikevich \cite{Izh-07}. We address here a stochastic Hodgkin-Huxley model including dendritic input, the latter being the only source of "noise". Practically we add as in \cite{HH-TD} and \cite{OU-HH} an autonomous stochastic differential equation as fifth component to the 4d model; its increments take the place of classical input terms and thus act on the membrane potential. Hence our stochastic Hodgkin-Huxley model is a 5d system 
$$
X=(X_t)_{t\ge 0} \quad\mbox{where $X_t$ has components $v_t$, $n_t$, $m_t$, $h_t$, $\xi_t$,} 
$$ 
driven by 1d Brownian motion $W=(W_t)_{t\ge 0}$ as follows:  
\beqq\label{xiHH_general}
\left\{\begin{array}{lll}
dv_t &= &d\xi_t \;-\; F(v_t, n_t, m_t, h_t)\, dt \\
dj_t &= &[\,\al_j(v_t)(1-j_t) \;-\; \beta_j(v_t)\, j_t\,]\, dt \quad,\quad j\in\{n,m,h\}   \\
d\xi_t &= &\beta(t,\xi_t)\, dt \;+\; q(\xi_t)\, dW_t. 
\end{array}\right.
\eeqq 
%\beqq\label{xiHH_general}
%\left\{\begin{array}{lll}
%dv_t &= &d\xi_t \;-\; F(v_t, n_t, m_t, h_t)\, dt \\
%dn_t &= &[\,\al_n(v_t)(1-n_t) \;-\; \beta_n(v_t)\, n_t\,]\, dt \\
%dm_t &= &[\,\al_m(v_t)(1-m_t) \;-\; \beta_m(v_t)\, m_t\,]\, dt \\
%dh_t &= &[\,\al_h(v_t)(1-h_t) \;-\; \beta_h(v_t)\, h_t\,]\, dt \\
%d\xi_t &= &\beta(t,\xi_t)\, dt \;+\; q(\xi_t)\, dW_t. 
%\end{array}\right.
%\eeqq
The mapping $(t,y)\mapsto\beta(t,y)$ is $T$-periodic in the time variable, for all $y$. Some $T$-periodic deterministic signal $t\to S(t)$ is coded in $\beta(t,y)$ and hence in the semigroup of $(X_t)_{t\ge 0}$. The mapping $y\mapsto q(y)$ is a 1d volatility, strictly positive on the interval where $(\xi_t)_{t\ge 0}$ takes its values. We shall specify $\beta(\cdot,\cdot)$ and $q(\cdot)$ below in two different --biologically relevant-- ways.

Without any change with respect to the deterministic Hodgkin-Huxley model, the function  
\beqq\label{function_F}
F(v,n,m,h) \;:=\; 36\, n^4\, (v+12) \;+\; 120\, m^3\, h\, (v-120) \;+\; 0.3\, (v-10.6)  
\eeqq 
in (\ref{xiHH_general}) is a power series in the four variables $v$, $n$, $m$, $h$. The mappings $v\to \al_j(v)>0$ and $v\to \beta_j(v)>0$ are real analytic having domain $\bbr$, the index $j$ representing any of the variables $n$ or $m$ or $h$ (see \cite{Izh-07} pp.\ 37--38 for explicit expressions --those which we have used in \cite{HH-TD}, \cite{OU-HH}-- and biological context). Since $\al_j(v)$ and $\beta_j(v)$ are strictly positive, solutions $\,t \to n_t \,,\, m_t \,,\, h_t\,$ starting in $[0,1]$ will immediately enter the open interval $(0,1)$ and remain there for all $0<t<\infty$.  This property allows  to construct  the state space $E$ such that parts b) and b') of Assumption \ref{ass:1} do hold.

Deterministic Hodgkin-Huxley systems with constant input $c\in\bbr$   
suppress the last equation from (\ref{xiHH_general}) and write the first four equations as 
\beqq\label{HH_det}
\left\{\begin{array}{lll}
dv_t &= &c\, dt \;-\; F(v_t, n_t, m_t, h_t)\, dt \\
dj_t &= &[\,\al_j(v_t)(1-j_t) \;-\; \beta_j(v_t)\, j_t\,]\, dt \quad,\quad j\in\{n,m,h\} \;. 
%dn_t &= &[\,\al_n(v_t)(1-n_t) \;-\; \beta_n(v_t)\, n_t\,]\, dt \\
%dm_t &= &[\,\al_m(v_t)(1-m_t) \;-\; \beta_m(v_t)\, m_t\,]\, dt \\
%dh_t &= &[\,\al_h(v_t)(1-h_t) \;-\; \beta_h(v_t)\, h_t\,]\, dt   
\end{array}\right.
\eeqq
In 4d systems (\ref{HH_det}), 
equilibria for the gating variables $n$, $m$, $h$ whenever $v$ is kept constant   
\beqq\label{nmhinfini}
j_\infty(v) := \frac{\al_j}{\al_j+\beta_j}(v)  \quad,\quad j\in\{n,m,h\}
\eeqq 
exist, and the mapping $F_\infty$ defined by 
\beqq\label{F_infinity}
v \;\lra\; F_\infty(v) \;:=\; F\left(\, v \,,\, n_\infty(v) \,,\, m_\infty(v) \,,\, h_\infty(v) \,\right)  
\eeqq 
turns out to be strictly increasing from $\bbr$ onto itself.  
Defining $v^c\in\bbr$ as the solution of  $F_\infty(v^c)\;=\; c$, we dispose of a 1-1-correspondence between values $c\in\bbr$ of constant input %in (\ref{HH_det}) 
and values $v^c\in\bbr$ of constant voltage such that 
\beqq\label{equilibrium-1}
\left(\, v^c \,,\, n_\infty(v^c) \,,\, m_\infty(v^c) \,,\, h_\infty(v^c) \,\right)
\eeqq
is an equilibrium for the system (\ref{HH_det}). Equilibria (\ref{equilibrium-1}) are stable or unstable depending on whether $c$ is below or above some critical value~$c^*$. Existence of this critical value has been proved by Rinzel and Miller \cite{RinMil-80}, for the original model constants of Hodgkin and Huxley \cite{HodHux-51} which are slightly different from Izhikevich \cite{Izh-07}.   
A numerical verification under the model constants of \cite{Izh-07} is given in Endler (\cite{End-12}, fig.\ 2.6 on p.\ 28); it locates the critical value at $c^*\approx 5.265$.  
In particular, the equilibrium point 
\beqq\label{equilibrium-2}
\left(\, v^0 \,,\, n_\infty(v^0) \,,\, m_\infty(v^0) \,,\, h_\infty(v^0) \,\right)
\eeqq
corresponding to $c=0$ is stable. Its value computed numerically is $v^0 \approx 0.0462$.

We shall specify the fifth component $(\xi_t)$ in (\ref{xiHH_general})  
in two ways: a CIR-type modelization and an OU-type one. We prove that in both cases all assumptions needed for Theorem \ref{theo:1} are satisfied. Both are biologically relevant: in 1d integrate-and-fire neuronal models, rescaled CIR or OU diffusions have been used for a long time to model the membrane potential, see \cite{LanSacTom-95} and the references therein. In contrast to this we use CIR or OU as a model for the dendritic input.

CIR-HH will denote the 5d system (\ref{xiHH_general}) for which the 5th equation 
takes the form 
\beqq\label{defCIR}
\left\{\begin{array}{l}
d\xi_t \;=\; [\,a + S(t) - \xi_t\,]\, dt \;+\; \sqrt{\xi_t\,}\, dW_t \;\;,\;\; t\ge 0 \;\;,\;\; \xi_0=\zeta>0  \\
2a\;>\; 1 \\
\mbox{$t\to S(t)$ nonnegative, real analytic and $T$-periodic} \;.  
\end{array}\right.
\eeqq

Assuming $2a\ge 1$ and $S(\cdot)\ge 0$, the process $(\xi_t)_{t\ge 0}$ %(\ref{defCIR}) 
starting from $\zeta>0$ almost surely never attains~$0$, and thus takes its values in the open half-axis $(0,\infty)$ on which the function $\,y\to \sqrt{y\,}$ is analytic. We shall consider the 5d CIR-HH process $X=(X_t)_{t\ge 0}$ defined by (\ref{xiHH_general})+(\ref{defCIR}) on the state space 
$$
E := \bbr \times [0,1]^3 \times (0,\infty) = \bigcup_m C_m \quad,\quad C_m={\rm cl}(G_m) 
\;\;\mbox{with}\;\; G_m := (-m , m)\times (0,1)^3 \times (\frac1m , m) \;. 
$$ 
Since $\al_j(\cdot)$ and $\beta_j(\cdot)$ in (\ref{xiHH_general}) are strictly positive,  the process $X$  starting in $C_m\setminus G_{m+1}$ enters $G_{m+1}$ immediately.  Thus Assumption \ref{ass:1} does hold; we can take $\,U := \bbr{\times}\bbr^3{\times}(0,\infty)\,$.

\begin{prop}\label{prop:1}
Assuming that $ 2 a > 1 ,$ Assumptions \ref{ass:2}, \ref{ass:3} and \ref{ass:5} are satisfied for CIR-HH with 
\beqq\label{CIRcandidate}
x^* \;=\;  \left(\, v^* ,  n^* , m^* , h^* , 1 \,\right) 
\;:=\; \left(\, v^0 \,,\, n_\infty(v^0) \,,\, m_\infty(v^0) \,,\, h_\infty(v^0)  \,,\, 1 \,\right) \;. 
\eeqq
\end{prop}

The proof for Proposition \ref{prop:1} is given in Section \ref{sec:5} below.  
We do not insist on scaling or on constants which are of interest for biologists (\cite{LanSacTom-95},  
\cite{DitLan-06}, \cite{Hoe-07}): in view of proving ergodicity of the system, this does not make any difference, so we simplify and work with the form (\ref{defCIR}) above.

OU-HH will denote the system (\ref{xiHH_general}) for which the 5th equation takes the form 
\beqq\label{defOU}
d\xi_t \;=\; (\,S(t)-\xi_t\,)\,  dt \;+\;  dW_t  \quad,\quad t\ge 0 \;. 
\eeqq 
We have considered in \cite{OU-HH} a parametrized equation in view of biologically realistic modelization; we simplify here to (\ref{defOU}) for the reason given above.  
We could prove in \cite{OU-HH} that OU-HH admits a finite number of Harris sets, and in restriction to every Harris set is positive Harris recurrent with an "explicit" invariant measure. In the present paper we are able to show that in fact a stronger assertion holds true: OU-HH  admits a unique Harris set, as a consequence of Theorem~\ref{theo:1} above. 
To show this we have to check the assumptions of Theorem~\ref{theo:1}.  
We consider the 5d process $X=(X_t)_{t\ge 0}$ defined by (\ref{xiHH_general})+(\ref{defOU}) on the state space  
$$
E \;:=\; \bbr \times [0,1]^3 \times \bbr \;=\; \bigcup\limits_m C_m \quad,\quad C_m={\rm cl}(G_m) \;\;\mbox{with}\;\; G_m := (-m,m)\times (0,1)^3 \times (-m,m) \;. 
$$
By the same argument as above, Assumption \ref{ass:1} holds, with $U:=\bbr \times \bbr^3\times \bbr$.

\begin{prop}\label{prop:2}
Assumptions \ref{ass:2}, \ref{ass:3} and \ref{ass:5} are satisfied for OU-HH with 
\beqq\label{OU-HH-equilibrium}
x^* \;=\;  \left(\, v^* ,  n^* , m^* , h^* , 0 \,\right) 
\;:=\; \left(\, v^0 \,,\, n_\infty(v^0) \,,\, m_\infty(v^0) \,,\, h_\infty(v^0)  \,,\, 0 \,\right) \;. 
\eeqq
\end{prop}

See Section \ref{sec:5} for the proof.
We insist on the fact that choice (\ref{OU-HH-equilibrium}) is different from our choice in  \cite{OU-HH} (where constant input $c\approx -0.0534$ was chosen such that the voltage $v^c=0$ equals zero, and where $v^c= ( 0,  n_\infty(0), m_\infty(0), h_\infty(0) )$ was considered in the first four components of $x^*$ which leads to different properties of the control, cf.\ \cite{OU-HH} section 2.4). 
We sum up the above discussions in the following theorem, immediate from Propositions \ref{prop:1} and~\ref{prop:2}.

\begin{theo}\label{theo:2} CIR-HH (under the assumption $2 a > 1 $) and OU-HH are positive Harris recurrent %Feller 
processes for which Theorem \ref{theo:1} and Corollary \ref{cor:1bis} hold.
\end{theo}

%%%%%%%%%%%%%%%%%%%%%%%%%%%%%%%%%%%%%%%%%%%%%%%%%
\subsection{Limit theorems: an empirical distribution function for the interspike times
%{\color{red} Some examples for limit theorems in stochastic Hodgkin Huxley models}
}\label{sec:limit theorems}
%%%%%%%%%%%%%%%%%%%%%%%%%%%%%%%%%%%%%%%%%%%%%%%%%5
 
Positive Harris recurrence allows to speak of spiking characteristics of the neuron in a sense 
of strong laws of large numbers. 
Interspike times are a notion of major interest in neuroscience. We stress that data (intracellular recording of the membrane potential in a cortical neuron in good time resolution) show that {\it the} time where a spike begins or {\it the} time where it ends cannot be identified from behavior of the variable $v$ taken alone, e.g.\ in form of thresholds for $v$ (which do not exist) or in form of other simple criteria based only on observation of $v$. The Hodgkin-Huxley model accounts for this since the gating variables $n,m,h$ are responsible for opening or closing of ion channels. During an interspike time, we observe $m < h$, during a spike we observe $m>>h$ (more exactly: $m$ close to $1$ and $h$ small), at the end of the spike we observe again $m < h$. The time interval on which $m < h$ holds includes both a kind of refractory period immediately following a spike and then a waiting time up to occurrence of a new spike. 
We shall characterize the spiking activity through stopping times defined in terms of $h$ and $m$. This is for large $t$ approximately equivalent to a characterization in terms of the past of $v$, as a consequence of  representations (41) in \cite{HH-TD} for the gating variables $j\in\{n,m,h\}$ 
\beqq\label{expl_expr_gating_var}
j_t \;=\; j_0\,e^{-\int_0^t (\al_j+\beta_j)(v_s)\,ds} \;+\; \int_0^t \al_j(v_r) \,e^{-\int_r^t (\al_j+\beta_j)(v_s)\,ds}\, dr \;\;,\;\; t\ge 0 
\eeqq
where the influence of the starting value vanishes as $t\to\infty$. This shows also  
that gating variables in a stochastic Hodgkin-Huxley model (\ref{xiHH_general}) have $\calc^1$-paths. 
Let $X_t=(X_t)_{t\ge 0}$ denote either CIR-HH or OU-HH as in the preceding subsection. Define events 
$$
F_{\rm sp} \;:=\; \{ x = (v,n,m,h,\zeta)\in E : \,m\,>\,h\, \} \quad,\quad F_{\rm b} \;:=\; \{ x = (v,n,m,h,\zeta)\in E : \,m\,<\,h\, \}
$$ 
with subscripts for `spike' or `between successive spikes', then introduce stopping times by $\si_0\equiv 0$ and %successively 
\beqq\label{def-tau_n}
\tau_n \;:=\; \inf\left\{ t>\si_{n-1} : X_t \in F_{\rm sp} \right\}   \;\;,\;\; 
\si_n \;:=\; \inf\left\{ t>\tau_n+\delta : X_t\in F_{\rm b} \right\} \;\;,\;\; n\ge 1  \;,
\eeqq
where we think of $\delta>0$ as a deterministic refractory period  
(this period, during which the neuron is not able to respond to any stimulus whatsoever as observed by biologists, cf.\ \cite{Izh-07}, provides a key tool for the proof of Theorem \ref{theo:4} below). 
The sequence $(\tau_n)_{n\in\bbn}$ marks on the time axis the beginning of successive spikes. It will be such that $\tau_n<\infty$ for every $n$ and $\tau_n\uparrow\infty$ as $\nto$, almost surely, as a consequence of the following theorem.

\begin{theo}\label{theo:cham}
For CIR-HH under the assumption $ 2 a > 1 $ or OU-HH, for all $ x \in E: $\\ % the following two statements hold true. \\
a) $Q_x ( \mbox{ there exists an infinite number of spikes } ) = 1.$\\
b) $Q_x ( \mbox{ there exists an infinite number of intervals $ [k T, (k+1) T [ $ on which no spike appears }) = 1.$ 
\end{theo}

The proof of Theorem \ref{theo:cham} follows from support properties established in \cite{HH-TD} and is given in Section~\ref{sec:5} below. 
Immediate applications of Theorem \ref{theo:1} and Corollary \ref{cor:1} a) to CIR-HH and OU-HH are the proportion of time spent spiking (or in $F_{\rm sp}$) by 
$$
\lim_{\nto}\; \frac{1}{\tau_n}\; \sum_{j=1}^n(\si_j-\tau_j) \;=\; \lim_{t\to\infty}\; \frac1t \int_0^t 1_{F_{\rm sp}}(X_s)\, ds  \;=\; \frac1T\int_0^T ds\, [\mu P_{0,s}](F_{\rm sp}) 
$$
or the specific shape of a spike by means of test functions $\psi$ through  
$$
\lim_{t\to\infty} \frac1t \int_0^t \psi(X^1_s)\, 1_{F_{\rm sp}}(X_s)\, ds  \;=\; \frac1T\int_0^T\int_E  ds\, [\mu P_{0,s}](dx)\, \psi(x^1)\, 1_{F_{\rm sp}}(x) 
$$
almost surely. As an essentially different application of strong laws of large numbers, far beyond (and not deducible from) Corollary \ref{cor:1} and based on analyticity of all coefficients in CIR-HH and OU-HH,  
we shall consider below the distribution function of the length of interspike intervals.  
Successive interspike times have no reason to be independent, and we may have single spikes as well as spike bursts (this follows from Theorem 5 of \cite{HH-TD} combined with the complex behaviour of deterministic Hodgkin-Huxley models). Positive Harris recurrence allows for the following Glivenko-Cantelli type result.

\begin{theo}\label{theo:4} 
For CIR-HH with $2a>1$ or OU-HH, consider empirical distribution functions for interspike times defined from the stopping times in (\ref{def-tau_n}):
$$
\wh F_n(t) \;:=\; \frac1n \sum_{j=1}^n 1_{[0,t]}(\tau_{j+1}-\tau_j) \quad,\quad t\ge 0   \;.  
$$
Then for every choice of a starting point $x \in E$, we have $Q_x$-almost surely as $\nto$ 
$$
\sup_{t\ge 0}\; \left|\, \wh F_n(t) \;-\; F(t) \,\right| \;\,\lra\;\; 0 \quad\mbox{as $n\to\infty$}
$$
where $F$ is a proper distribution function.  $\,F(t)$ has the interpretation of a relative number of expected interspike times smaller than $t$ over "typical" life cycles of $\ov X$.  
\end{theo}

Here "typical" refers to life cycles which we can construct by Nummelin splitting, cf.\  \cite{Num-78}.  
The proof of Theorem~\ref{theo:4}, given in Section \ref{sec:5},  
works thanks to analyticity of the coefficients.

%%%%%%%%%%%%%%%%%%%%%%%%%%%%%%%%%%%%%%%%%%%%%%%%%%%%%%%%%%%
\section{Control systems. Weak Hoermander dimension} 
%%%%%%%%%%%%%%%%%%%%%%%%%%%%%%%%%%%%%%%%%%%%%%%%%%%%%%%%%%%
In the sequel, given an SDE $dY_t=\delta(t,Y_t)dt+\Sigma(Y_t)dB_t$ driven by Brownian motion $B$ in the Ito sense, we have to pass to its Statonovitch form $dY_t \;=\; \wt \delta(t,Y_t)\, dt \;+\; \Sigma(Y_t) \circ dB_t$ with Stratonovich drift 
\beqq\label{strato}
\wt \delta ^i (t,y) \;=\;  \delta ^i (t,y)  \;-\;  \frac12\sum_{\ell=1}^m  \sum_{j=1}^d \Sigma^{j,\ell}(x)\frac{\partial \Sigma^{i,\ell}}{\partial x^j}(x) \quad,\quad 1\le i\le d \; 
\eeqq
 (cf. Kunita \cite{Kun-90} p.\ 60,  Bass \cite{Bas-98} p.\ 198-199).

%%%%%%%%%%%%%%%%%%%%%%%%%%%%%%%%%%%%%%%%%%%%
\subsection{Control systems: extension of the support theorem}\label{sec:control1}
%%%%%%%%%%%%%%%%%%%%%%%%%%%%%%%%%%%%%%%%%%%

The control theorem goes back to Strook and Varadhan \cite{StrVar-72}. We quote\footnote{
%anfangfussnote
time-dependent drift does not alter the structure of proof in \cite{MilSan-94}; see also remark 2.2 and (4.2)-(4.3) in \cite{StrVar-72}.
%endefussnote
} it in the form Millet and Sanz-Sole \cite{MilSan-94}, theorem 3.5.  Consider an SDE $dY_t=\delta(t,Y_t)dt+\Sigma(Y_t)dB_t$, with state space $E=\bbr^d$, $m\le d$ is the dimension of the driving Brownian motion $(B_t)$. All components of $\Sigma(\cdot)$ are $\calc^2$ on $\bbr^d$, bounded with bounded derivatives of orders $1$ and $2$, and $\delta(\cdot,\cdot)$ is globally Lipschitz and bounded on $\bbt\times\bbr^d$ (cf. condition (H) of \cite{MilSan-94}).  

For time horizon $t_0<\infty$ which is arbitrary but fixed, write $\,\tt H\,$ for the Cameron-Martin space of measurable functions ${\tt h}:[0,t_0]\to \bbr^m$ having absolutely continuous components ${\tt h}^\ell(t) = \int_0^t \dot{\tt h}^\ell(s) ds$ with $\int_0^{t_0}[\dot{\tt h}^\ell]^2(s) ds < \infty$, $1\le \ell\le m$. For $x\in\bbr^d$ and ${\tt h}\in{\tt H}$, consider the deterministic system 
\beqq\label{generalcontrolsystem}
\vph = \vph^{({\tt h}, x)} \quad\mbox{solution to}\quad d \vph (t) = \wt \delta ( t, \vph (t) ) dt + \Sigma( \vph (t) ) \dot{\tt h}(t) dt, \quad\mbox{with starting point $\vph(0)=x$.}  
\eeqq
 Thus $\vph$ is a function $[0,t_0]\to \bbr^d$. The control theorem states that in restriction to finite time horizon $t_0$, the support of the law of $(Y_t)_{0\le t\le t_0}$ with starting point $Y_0=x$ coincides with the closure in $C([0,t_0],\bbr^d)$ of the set of control paths 
$$
A_x \;:=\; \left\{  \left( \vph^{({\tt h}, x)}(t) \right)_{0\le t\le t_0} : \,{\tt h}\in{\tt H}\, \right\} \;. 
$$
This follows from approximation of Stratonovich integrals by adapted polygonal interpolation of the driving Brownian path, and from Girsanov theorem. Polygonal interpolation means that there is some --sufficiently fine-- finite partition $0=s_0 < s_1 < \ldots < s_\nu = t_0$ such that all components  $\dot{\tt h}^\ell$ in (\ref{generalcontrolsystem}) remain constant between $s_{r-1}$ and $s_r$. Such controls ${\tt h}$ are called {\em admissible} by Arnold and Kliemann \cite{ArnKli-87}; in particular the support of the law of $(Y_t)_{0\le t\le t_0}$ starting from $Y_0=x$ coincides with the closure in $C([0,t_0],\bbr^d)$ of the following subset $\wt A_x$ of $A_x$: 
$$ 
\wt A_x \;:=\; \left\{  \left( \vph^{({\tt h}, x)}(t) \right)_{0\le t\le t_0} : \,{\tt h}\in{\tt H}\;\;\mbox{admissible}\; \right\} \;. 
$$

Extending this result to processes $X$ with state space $E\subset\bbr^d$ according to Assumption \ref{ass:1}, by using localization, we obtain the following result. 

\begin{theo}\label{theo:4bis}
Grant Assumption \ref{ass:1}. Denote by $ Q_x^{t_0} $ the law of the solution $ (X_t)_{0 \le t \le t_0} $ of  \eqref{sde}, starting from $X_0 = x .$ Let $\,\vph = \vph^{({\tt h},x)}\,$ (whenever it exists) denote a solution to
\beqq\label{controlsystem}
d \vph (t) \;=\; \wt b ( t, \vph (t) )\, dt \;+\; \si( \vph (t) )\, \dot{\tt h}(t)\, dt   \quad,\quad \vph(0)=x.
\eeqq
Then the following two assertions hold true.   \\
a) Fix $ 0 < t_0 < \infty ,$ $ x \in E  $ and  $  {\tt h } \in \tt H $ such that $\,\vph = \vph^{({\tt h},x)}\,$ exists on some time interval $ [ 0, \wt T ] $ for $\wt T > t_0$ and takes values in ${\rm int}(E)$ on $]0,\wt T]$. Then
$$ \left(\vph^{({\tt h} , x ) }\right)_{| [0,  t_0 ] } \in \overline{ {\rm supp} \left( Q_x^{t_0 } \right)}.$$

b) 
For $\,\dot{\tt h}:[0,\infty)\to\bbr^m\,$ piecewise constant and without accumulation of jumps in finite time and for all $x\in E$, $\,\vph = \vph^{({\tt h},x)}\,$ exists and takes values in $E$ on some interval 
$$
0\le t<s({\tt h}, x) \quad,\quad 0< s({\tt h}, x)\le\infty \;. 
$$
%and takes values in ${\rm int}(E)$ when $t>0$. 
Then for all $0<t_0<\infty$ and $x\in E$, the support of $Q_x^{t_0} $ %the law of the solution $(X_t)_{0\le t\le t_0}$ to equation (\ref{sde}) starting from $X_0=x$
is contained in the closure of 
\beqq\label{newcontrolsystems}
\left\{  \left( \vph^{({\tt h}, x)}(t) \right)_{0\le t\le t_0} : \,{\tt h}\in{\tt H}\;\;\mbox{admissible and satisfying $s({\tt h}, x)>t_0$}\; \right\} 
\eeqq
in the sense of uniform convergence in $C([0,t_0],E)$. 
\end{theo}

This theorem will be proved in the Appendix Section \ref{sec:app}. 
To our best knowledge this localized version of the control theorem adresses a new type of problem (completely different e.g.\ from control in presence of reflexion, see \cite{DosPri-82} and \cite{Pet-96}), arising from the link between topological properties of the state space and properties of the process formulated in Assumption \ref{ass:1} which itself is imposed by the structure of Hodgkin-Huxley equations.

%%%%%%%%%%%%%%%%%%%%%%%%%%%%%%%%%%%%%%%%%%%%%%%%%%%%%%%%%%%%
\subsection{Weak Hoermander condition}\label{sec:hoer}
%%%%%%%%%%%%%%%%%%%%%%%%%%%%%%%%%%%%%%%%%%%%%%%%%%%%%%%%%%%%%%%

We grant Assumption \ref{ass:1} and put the SDE \eqref{sde} in Stratonovich form $$
dX_t \;=\; \wt b(t,X_t)\, dt \;+\; \si(X_t) \circ dW_t.
$$
Recall that all coefficients of equation \eqref{sde} are $C^\infty-$functions on $ \bbr_+ \times U$ for some open set $ U \subset \bbr^d $ such that $ E \subset U .$  
Consider the process $\ov X=((i_T(t),X_t))_{t\ge 0}$ as in Section \ref{sec:setting}, write $\ov U = \bbt\times U $ and $ \ov E = \bbt \times E.$ 
In view of a representation $\,d\ov X_t = V_0(\ov X_t)dt + \sum_{l=1}^m V_\ell(\ov X_t)\circ d W^\ell_t\,$ for the process $\ov X$  
in terms of the components $W^1,\ldots,W^m$ of the driving Brownian motion in equation (\ref{sde}), we define  
(cf.\ Section 3.1 in \cite{HH-TD})   
vector fields on $\ov U $ taking values in $\bbr^{1+d}$ by 
$$
V_0(t,x) := \left( \begin{array}{l}  
\quad1 \\ \wt b ^1(t,x) \\ \quad\vdots \\ \wt b ^d(t,x) 
\end{array}\right)    
\quad,\quad 
V_\ell(t,x)  := \left( \begin{array}{l}  
\quad0 \\ \si ^{1,\ell}(x) \\ \quad\vdots \\ \si ^{d,\ell}(x) 
\end{array}\right) 
\;,\;   1\le \ell\le m \;,  
$$
that we identify with the first order differential operators 
$$
V_0 \;=\; \frac{\partial}{\partial t} + \sum_{j=1}^d \wt b^j(t,x)\frac{\partial}{\partial x^j} 
\quad,\quad 
V_\ell \;=\;   \sum_{j=1}^d \si^{j,\ell}(x)\frac{\partial}{\partial x^j} \;,\;   1\le \ell\le m \;. 
$$
Given a vector field $L:\ov U\to \bbr^{1+d}$ we denote by $L^i$ its components for $i=0$ or $i=1,...,d$. 
Whenever $L$ has $0$-component 
$L^0\equiv 0$, we have $[V_0,L]^0 \equiv [V_\ell,L]^0\equiv  0$ for $1\le \ell\le m$ and  
$$
[V_0,L]^i = \frac{\partial L^i}{\partial t} + \sum_{j=1}^d \left( V_0^j\, \frac{\partial L^i}{\partial x^j} - L^j\, \frac{\partial V_0^i}{\partial x^j}\right)  
$$
$$
[V_\ell,L]^i =  \sum_{j=1}^d \left( V_\ell^j\, \frac{\partial L^i}{\partial x^j} - L^j\, \frac{\partial V_\ell^i}{\partial x^j}\right)  
$$
for $i=1,...,d$, with the usual notation for the Lie bracket of two vector fields $[A,B]:= AB-BA$.

\begin{defin}\label{def:hoerm1} 
For fixed $N\in\bbn$, define a set $\call_N$ of vector fields by `initial condition' $V_1,\ldots,V_m \in \call_N$ and at most $N$ iteration steps 
\beqq\label{*}
L\in\call_N \;\Lra\; [L,V_0] , [L,V_1] , \ldots , [L,V_m] \in \call_N  \;.   
\eeqq
Write $\call_N^*$ for the closure of $\call_N$ under Lie brackets; its linear hull $\,{\rm span}(\call_N^*)$ is 
the Lie algebra ${\rm LA}(\call_N)$ spanned by $\call_N$. We define $\,\call := {\rm LA}(\bigcup_N \call_N)\,$ and $\,\Lambda := {\rm LA}(V_0, V_1, \ldots, V_m)\,$.  
\end{defin}

Note that all elements of $\call_N^*$ have $0$-component equal to zero, so $\,d\,$ is an obvious upper bound for 
${\rm dim\, span}(\call_N)$,  ${\rm dim}({\rm LA}(\call_N))$ and ${\rm dim}(\call)$ 
on $\ov E$. Now we can give the definition of full weak Hoermander dimension in Assumption~\ref{ass:3}:

\begin{defin}\label{def:hoerm2}
We say that a point $x^*\in U\supset E$ is of {\em full weak Hoermander dimension} if there is some $N\in\bbn$ such that   
\beqq\label{fwHd}
{\rm dim\, span}(\call_N) 
(s,x^*) \; =\; d \quad\mbox{independently of $s\in\bbt$} \;. 
\eeqq
\end{defin}

\vskip0.3cm 
Up to uniformity in time on the torus, this is the usual form of the definition (cf.\ e.g.\ Hairer \cite{Hairer2011} Def.\ 1.2).  We have the following.

\begin{prop}\label{prop:4}
Under Assumptions 1 and 2 a), if points $x^*$ in ${\rm int}(E)$ of full weak Hoermander dimension exist, then open neighborhoods $\,\wt U$ of $x^*$ exist in ${\rm int}(E)$ such that 
$$ 
{\rm dim\, span}(\call_N)(s,x) \;=\;  {\rm dim}(\call)(s,x)  
\;=\; d \quad\mbox{for all $(s,x)\in \bbt{\times}\wt U$} \;.
$$
\end{prop}

\begin{proof}{\bf of Proposition \ref{prop:4}}
\quad 
1) In analogy to Definition \ref{def:hoerm1}, consider for fixed $N\in\bbn$ a set of vector fields $\Lambda_N$ by 'initial condition' $V_0, V_1, \ldots,V_m$ in $\Lambda_N$ and at most $N$ iteration steps 
$$
L\in\Lambda_N \;\Lra\; [L,V_0] , [L,V_1] , \ldots , [L,V_m] \in \Lambda_N  \;.   
$$
Clearly $\Lambda_N\supset\call_N$ for all $N$. By Proposition 1 in \cite{HH-TD} and its proof, 
\beao
{\rm dim\, span}(\Lambda_N(t,x))  &=&  1 \;+\; {\rm dim\, span}(\call_N(t,x))   \\
{\rm dim}(\Lambda(t,x))   &=&   1 \;+\; {\rm dim}(\call(t,x))
\eeao
for all $(t,x)\in \ov U$. This is a consequence of the difference in the initialization between $\Lambda_N$ and $\call_N$, 
since time dependence in equation (\ref{sde}) occurs in the drift only. 
\\
2) If a point $x^*\in U$  of full weak Hoermander dimension exists, fix $x^*$ and $N$ as in Definition \ref{def:hoerm2}. By~1), $\,{\rm span}(\Lambda_N(t,x^*))$ has maximal dimension $d+1$ for all $t\in\bbt$. As a consequence, at $(t,x^*)$ for all $t\in\bbt$,  $\,{\rm span}(\Lambda_N)$ coincides with the full Lie algebra $\Lambda={\rm LA}(V_0,V_1,\ldots,V_m)$ viewed as a vector space. Then again by 1), $\,{\rm span}(\call_N)$ coincides with $\call = {\rm LA}(\bigcup_N\call_N)$ at $(t,x^*)$ for all $t\in\bbt$,  and then also with ${\rm span}(\call_N^*)={\rm LA}(\call_N)$ .  
\\
3) Assume now that points of full weak Hoermander dimension exist in ${\rm int}(E)$, in the sense of Definition \ref{def:hoerm2}. For every $\nu$ fixed, sets $\{ (s,y)\in \bbt{\times}{\rm int}(E) : {\rm dim\, span} (\call_N)(s,y) > \nu \}$ are open. Since dimension $\,d\,$ in (\ref{fwHd}) is a maximal choice, the set $\{ (s,y)\in \bbt{\times}{\rm int}(E) : {\rm dim\, span}(\call_N)(s,y) = d \}$ is open.  
\end{proof}

\begin{rem} 
Proposition \ref{prop:4} and step 1) of its proof show that Definition {\bf (LWH)} which we have used in \cite{HH-TD} coincides with Definition \ref{def:hoerm2} above up to uniformity in time on the torus.  
\end{rem}

{\bf Example \ref{ex:1} continued}\quad 
{\it Consider again the 2-dim process (\ref{toyexample}) driven by 1-dim Brownian motion on the state space $E=\bbr\times[0,\infty)$: then the point $x^*:= (0,\frac23)$, attainable in a sense of deterministic control for the process (\ref{toyexample}) by Section \ref{sec:setting}, is of full weak Hoermander dimension if and only if the constant $\,c>0$ in equation (\ref{toyexample}) satisfies $c<\frac32$. Indeed, write $x=({\xi \atop \psi})$ for points in $E$, $\xi\in\bbr$, $\psi\ge 0$. With $\wt b(t,x) = ( {-c\, \sin^2(2\pi t)\, \xi \atop 1 - \frac32 \psi} )$ as in Section \ref{sec:setting} we have to consider  vector fields $\ov U \to \bbr^{1+d}$ 
$$
V_0 = \left(\begin{array}{l} 1 \\ -c\, \sin^2(2\pi t)\, \xi \\ 1 - \frac32 \psi \end{array}\right) \;,\; 
V_1 = \left(\begin{array}{l} 0 \\ 1 \\ \psi \end{array}\right) \;,\; 
[V_0,V_1] = \left(\begin{array}{l} 0 \\ c\, \sin^2(2\pi t)  \\ 1 \end{array}\right) \;.
$$
For $x^*=(0,\frac23) \in {\rm int}(E)$ and $t\in\bbt$ this simplifies to 
$$
V_0(t,x^*) = \left(\begin{array}{l} 1 \\ 0 \\ 0 \end{array}\right) \;,\; 
V_1(t,x^*) = \left(\begin{array}{l} 0 \\ 1 \\ \frac23 \end{array}\right) \;,\; 
[V_0,V_1](t,x^*) = \left(\begin{array}{l} 0 \\ c\, \sin^2(2\pi t)  \\ 1 \end{array}\right) \;. 
$$
Thus $V_1$ and $[V_0,V_1]$ are linearly independent at $(t,x^*)$ for all $t$ if and only if 
$c<\frac32$. In this case we have $\,{\rm dim\, span}(\call_N)=2=d\,$ on $\bbt\times\{x^*\}$ already for $N=1$. 
}

%%%%%%%%%%%%%%%%%%%%%%%%%%%%%%%%%%%%%%%%%%%%%%%%%%%%%%%%%%%%
\subsection{Role of the attainable point}\label{sec:point}
%%%%%%%%%%%%%%%%%%%%%%%%%%%%%%%%%%%%%%%%%%%%%%%%%%%%%%%%%%%%%%%
If we specialize the control argument in Theorem \ref{theo:4bis} and focus only on time points which are multiples of the periodicity, relevant for the skeleton chain  $(X_{jT})_{j\in\bbn_0}$, we have the following: 

\begin{cor}\label{cor:2}
Grant Assumptions \ref{ass:1}, \ref{ass:2} a) and \ref{ass:3}. 
For arbitrarily small neighborhoods $U^*$ of the attainable point $x^*$ in ${\rm int}(E)$ and for arbitrary starting points $x\in E$, we have:    
$$
Q_x \left( \,X_{jT}\in U^* \right) \;\;>\;\; 0  \quad\mbox{for large enough $j\in\bbn$} \;. 
$$
\end{cor}
\begin{proof}
To see this, choose $0<t_0<\infty$ arbitrarily large but finite and a control ${\tt h}\in{\tt H}$, the Cameron-Martin space associated to $t_0$, such that the control path  $\vph^{({\tt h}, x)}:[0,t_0]\to E$ connects initial value $x$ to some terminal value in $U^*$. We can do this such that $t\to \vph^{({\tt h}, x)}(t)$ stays in $U^*$ over the time interval $[\frac12 t_0,t_0]$, cf.\ Definition \ref{def:1}. Then apply Theorem \ref{theo:4bis} and restrict to times $t=jT$ in $[\frac12 t_0,t_0]$. 
\end{proof}

In the above result $j$ may depend on $x$. In order to overcome this problem we introduce for fixed $ p \in ]0, 1 [ $ the transition kernel
\beqq\label{defres}
R ( x, dy) =  ( 1 - p ) \sum_{ k \geq 1} p^{k-1} P_{0, k T } ( x, dy ) \;, 
\eeqq
corresponding to sampling the skeleton chain $(X_{kT})_{k\in\bbn_0}$ at 
independent geometric times.
Let $(Z_\ell)_{ \ell \in \bbn _0} $ be the corresponding Markov chain. In order to prove positive Harris recurrence of the skeleton chain 
$(X_{kT})_k$
it is sufficient to prove positive Harris recurrence of the sampled chain
$(Z_\ell)_\ell . $

The following result is a main ingredient of the proof of Theorem \ref{theo:1}. Its proof is given in Section \ref{sec:harris}.

\begin{lem}\label{lem:1.5}
Grant Assumptions \ref{ass:1}--\ref{ass:3}. With $x^*$ from Assumption \ref{ass:3} both attainable and of full weak Hoermander dimension, fix some open set $\wt U \subset{\rm int}(E)$ containing $x^*$ such that the weak Hoermander condition holds on ${\wt U}$. If $U^{**} \subset \wt U  $ is an arbitrary neighborhood of $x^* ,$ we can find some ball $C^*:=B_{\vep^*} (x^*)$  and some point $y^*\in U^{**}$ with $ B_{\vep^* } ( y^* ) \subset U^{**} $ such that writing $\nu^*$ for the uniform law on $B_{\vep^*} (y^*)$, Nummelin's  minorization condition holds (see \cite{Num-78} and \cite{Num-85}):
\begin{equation}\label{nummelin}
R(x', dy') \;\ge\; \al^*\; 1_{C^*}(x')\; \nu^*(dy') \quad,\quad x',y' \in E.
\end{equation}
\end{lem}
\vskip0.5cm

We would like to stress that Lemma \ref{lem:1.5} provides a substantial improvement over the type of result which we obtained in \cite{OU-HH}. We worked there with a finite set of splitting conditions, see 
formulae (16)--(17) and Theorem 3 of \cite{OU-HH}, 
and could establish the existence of a finite collection of disjoint Harris sets. In contrast to this, our route now works for a broader class of processes and establishes the existence of a unique Harris set.

%%%%%%%%%%%%%%%%%%%%%%%%%%%%%%%
\subsection{The analytic case}\label{sec:control}
%%%%%%%%%%%%%%%%%%%%%%%%%%%%%%%%%%%%

Throughout this section, we grant Assumptions \ref{ass:1}, \ref{ass:2} a) and \ref{ass:5}.   
In order to control $\ov X = ((i_T(t),X_t))_{t\ge 0}$ of Section \ref{sec:setting},   
we view the functions $\vph^{({\tt h}, x)}$ appearing in  Theorem \ref{theo:4bis} as $(1{+}d)$-dimensional functions 
$$
\left[0,s({\tt h}, x)\right) \;\ni\;\; t \;\lra\;  
\bar\vph^{({\tt h}, x)}(t) \;:=\;  \left(  i_T(t) ,  \vph^{({\tt h}, x)}(t)  \right) %\;\; \in\;  \ov E
$$
taking values in $\ov E = \bbt{\times}E$ where $\bbt=[0,T]$ is the torus, with initial value $\left( 0 , x \right)\in \ov E$. 
If ${\tt h}\,$ is admissible as in (\ref{newcontrolsystems}) with constant components $\,\dot{\tt h}^\ell(\cdot)\equiv\gamma_{\ell,r}\,$, $1\le\ell\le m$, on  intervals $]s_{r-1},s_r]$,  the trajectory $\,t\to\bar\vph^{({\tt h}, x)}(t)\,$ moves between times $s_{r-1}$ and $s_r$ along the vector field 
$$ 
V_0 \;+\; \gamma_{1,r} V_1 + \ldots + \gamma_{m,r} V_m   
$$
where $V_0$ and $V_\ell$, $1\le\ell\le m$, have been defined in Section \ref{sec:hoer}. In this case $\ov\vph^{({\tt h}, x)}$ is a piecewise integral curve  $\,t\to\bar\vph^{({\tt h}, x)}(t)\,$ of ${\rm span}(V_0,V_1,\ldots, V_m)$ in terms of Sussmann \cite{Sus-73}.

\begin{lem}\label{lem:1}  
Grant Assumptions \ref{ass:1}, \ref{ass:2} a) and \ref{ass:5}. Then for $0<t_0<\infty$ arbitrary but fixed, the dimension of the Lie algebra ${\rm LA}(V_0,V_1,\ldots, V_m)$ remains constant along  
$$
\bar\vph^{({\tt h}, x)} = \left(  i_T(\cdot) ,  \vph^{({\tt h}, x)} \right) \; : \; [0,t_0]\;\lra\; \ov E
$$
if $\vph^{({\tt h}, x)}$ belongs to the set~(\ref{newcontrolsystems}), i.e.\ if $\,{\tt h} \in {\tt H}$ is admissible and satisfies  $\,{\tt s}({\tt h},x)>t_0\, .$ 
\end{lem}

\begin{proof}
Write for short $\cals := {\rm span}(V_0,V_1,\ldots, V_m)$,  $\;\cals^*$ for the closure of $\cals$ under Lie brackets, and following Sussman \cite{Sus-73}, we write $\Delta_{\cals^*}= {\rm LA}(V_0,V_1,\ldots, V_m)$ for the linear space spanned by $\cals^*$. 
\\
1) Sussmann \cite{Sus-73} defines $\cals$-orbits in $\ov E$ as equivalence classes of points in $\ov E$ under the following equivalence relation: $(s,y)$ and $(s',y')$ are equivalent if there is a piecewise integral curve of $\cals$ connecting $(s,y)$ and $(s',y')$, forward or backward in time. If we focus on a control function as in~(\ref{newcontrolsystems}) 
$$
\mbox{$\vph^{({\tt h}, x)} : [0,t_0]\to E$ with $x\in E$ and  $\,{\tt h}\in{\tt H}\,$ admissible and such that $s({\tt h}, x)>t_0$}
$$
we see that all points in the graph of $\bar\vph^{({\tt h}, x)}$ are equivalent and thus belong to the same  
$\cals$-orbit in $\ov E$. 
\\
2) By Assumption \ref{ass:5}, $\,\cals$ is a set of real analytic vector fields on $\ov U=\bbt\times U$. Thus by Sussmann \cite{Sus-73} Section~9, the closure $\,\cals^*$ of $\cals$ under Lie brackets is "locally of finite type". 
By Theorem 8.1 in \cite{Sus-73} and the lines following its proof, this property implies that $\Delta_{\cals^*}$ is $\cals^*$-invariant,  
i.e.\ $\Delta_{\cals^*}$ coincides with $\Delta^{\cals^*}_{\cals^*}=\calp_{\cals^*}$ in notation of \cite{Sus-73}. 
The dimension of $\calp_{\cals^*}$ remaining constant along $\cals^*$-orbits (cf.\ pp.\ $177^{9-11}$ and  $176_{10-6}$ in \cite{Sus-73}), 
it follows that the dimension of $\,\Delta_{\cals^*}$ remains constant along $\cals^*$-orbits.  
In particular,   
$$%\beqq\label{dimconstantoncontrolpaths}
s \;\;\lra\;\; {\rm dim}\, \Delta_{\cals^*} ( \bar\vph^{({\tt h}, x)}(s) ) \;=\; {\rm dim}\; {\rm LA}(V_0,V_1,\ldots, V_m) \left(  i_T(s) , \vph^{({\tt h}, x)}(s)  \right)
$$%\eeqq
remains constant on $[0,t_0]$ for all $x\in E$ and all $\tt h$ as in (\ref{newcontrolsystems}).  
\end{proof}

We strengthen that the assertion of Lemma \ref{lem:1} hinges on  Assumption \ref{ass:5}, and  
would not hold true under our standing $\calc^\infty$-assumption (a counterexample is given in \cite{Sus-73} p.\ $177$). Now we exploit Assumption \ref{ass:3}: combining the results of Section \ref{sec:hoer} with Lemma \ref{lem:1}, we can prove  

\begin{theo}\label{theo:6} 
Under Assumptions \ref{ass:1}, \ref{ass:2} a), \ref{ass:3} and \ref{ass:5}, all points $\,x\,$ in the state space $E$ are of full weak Hoermander dimension. 
\end{theo}

\begin{proof}\quad 
According to Assumption \ref{ass:3}, fix $x^*$ in ${\rm int}(E)$ such that $x^*$ is attainable in a sense of deterministic control, see Definition \ref{def:1}, and of full weak Hoermander dimension, see Definition~\ref{def:hoerm2}.  
\\ 
1) Consider the sets of vector fields $\,\Lambda={\rm LA}(V_0,V_1,\ldots, V_m)$ and $\,\call={\rm LA}(\bigcup_N\call_N)$ on $\ov U=\bbt\times U$,  with notations of Section \ref{sec:hoer}. We have  $\,{\rm dim}(\Lambda) = {\rm dim}(\call)+1\,$ on $\ov U$ 
by step 1) of the proof of Proposition~\ref{prop:4}.  
\\
2)  Proposition \ref{prop:4} shows that $\,{\rm dim}(\call)(s,x)=d\,$ for all $(s,x) \in \bbt\times\wt U$ where $\wt U\subset {\rm int}(E)$ is open and contains $\{x^*\}$. Thus from 1), 
$
{\rm dim}\, {\rm LA}(V_0,V_1,\ldots, V_m)(s,x) = d+1
$
for all $(s,x) \in \bbt\times\wt U$. 
\\
3) By Definition \ref{def:1},  
for all starting points $x$ in $E$ and arbitrarily small neighborhoods $U^*$ of $x^*$ in ${\rm int}(E)$, we can find some time horizon $0<t_0<\infty$ and some control ${\tt h}\in{\tt H}$, the Cameron-Martin space associated to $t_0$, such that the control path  $\bar\vph^{({\tt h}, x)}:[0,t_0]\to \ov E$ connects an initial value $(0,x)$ to some terminal value in $\bbt\times U^*$. Here we have degrees of freedom in choosing $t_0$ and $U^*$, hence the same property remains true in restriction to admissible controls $\tt h$. The same can be done with initial values~$(s,x)\in \ov E$. 
\\ 
4) Along the control paths  $\,\bar\vph^{({\tt h}, x)}:[0,t_0]\to \ov E\,$ of 3) which drive $(s,x)$ into $\bbt\times U^*$, with  ${\tt h}\in {\tt H}$  admissible and $t_0$ large enough, $\,{\rm dim}\, {\rm LA}(V_0,V_1,\ldots, V_m)$ remains constant by Lemma \ref{lem:1}. This establishes $\,{\rm dim}\, {\rm LA}(V_0,V_1,\ldots, V_m)(s,x) = 1+d\,$ for all $(s,x)\in \ov E$. Using again 1), we have  $\,{\rm dim}(\call) (s,x) = d\,$  for all $(s,x)\in \ov E$.  
\end{proof}

%%%%%%%%%%%%%%%%%%%%%%%%%%%%%%%%%%%%%
\section{Proofs for Section \ref{sec:results} 
%Proofs of Theorems \ref{theo:1} and \ref{theo:2}
} 
%%%%%%%%%%%%%%%%%%%%%%%%%%%%%%%%%%%%

%%%%%%%%%%%%%%%%%%%%%%%%%%%%%%%%%%%%%%%%%%%%
\subsection{Proof of Theorem \ref{theo:1}. Local Lebesgue densities.
%\ref{theo:2} and \ref{theo:cham}
}\label{sec:harris}
%%%%%%%%%%%%%%%%%%%%%%%%%%%%%%%%%%%%%%%%%%%

We now give the proofs for Lemma \ref{lem:1.5}  
(which relies on Lemma \ref{lem:2} and Corollary \ref{cor:4bis} below)  
and Theorem~1. We quote the following from \cite{HH-TD}.

\begin{lem}\label{lem:2}
Under Assumptions \ref{ass:1}--\ref{ass:3} the following holds true. There exists an open neighborhood $\wt U  $ of $x^*$  $in\; {\rm int}(E) $ such that for $0<t<\infty$, transition probabilities $P_{0,t}(x,dy)$ of the process (\ref{sde}) admit Lebesgue densities $p_{0,t}(x,y)$ locally on $ y \in \wt U  $ with the following properties: $\wt U  \ni y\to p_{0,t}(x,y)$ is  
infinitely differentiable   
when $x\in E$ is fixed, and $x\to p_{0,t}(x,y)$ is lower semicontinuous on $E$ when $y\in  \wt U $ is fixed. Moreover, for all measurable sets  $A \subset { \wt U }$ and all $t$, the mapping $x \to P_{0,t}(x,A)$ is lower semicontinuous on $E$.
\end{lem}

\vskip-0.5cm
\begin{proof} 
Our Assumptions \ref{ass:1}--\ref{ass:3} and Proposition \ref{prop:4} imply {\bf (H1)}, {\bf (H2)} and {\bf (LWH)} of \cite{HH-TD}; the assertion is from \cite{HH-TD}, Theorems 1 and 2. The proof of Theorem 2 in \cite{HH-TD} shows moreover that for all measurable sets  $A \subset  \wt U $ and all $t$, the mapping $x \to P_{0,t}(x,A)$ is lower semicontinuous on $E$.  
\end{proof}

\begin{cor}\label{cor:4bis}
Under the assumptions of Lemma \ref{lem:2}, with $\wt U$ as there, the following holds true. \\
a) The transition kernel $R$ defined in (\ref{defres}) admits transition densities $ r (x, y ) $ locally in $ y \in \wt U  $ with the following properties: $\wt U  \ni y\to r (x,y)$ is  lower semicontinuous when $x\in E$ is fixed, and $x\to r(x,y)$ is lower semicontinuous  on $E$ when $y\in \wt U $ is fixed. Moreover, for any measurable set $A \subset \wt U , $ $x \to R( x, A ) $ is lower semicontinuous on $E$. \\
b) For arbitrarily small neighborhoods $ U^* \subset \wt U  $ of the attainable point $x^* $ of Assumption \ref{ass:3} and for arbitrary starting points $x \in E,$ we have that $ R ( x, U^* ) > 0 .$ \\
c) The sampled chain $ (Z_\ell)_{ \ell \in \bbn _0} $ is a $T-$chain in the sense of Meyn and Tweedie \cite{MeyTwe-92}.
\end{cor}
\vskip-0.5cm
\begin{proof}
Assertion a) is a direct consequence of Lemma \ref{lem:2} and of the definition of $R.$ b) follows from Corollary \ref{cor:2}. In order to prove c), observe that $R (x, \cdot \cap U^* ) $ is a continuous component of $ R $ in the sense of \cite{MeyTwe-92}, page 548, due to a), which is non-trivial due to b).
\end{proof}

\vskip0.3cm
{\bf Proof of Lemma \ref{lem:1.5} } With $\wt U$ from Lemma \ref{lem:2} and Corollary \ref{cor:4bis}, we consider an arbitrarily small neigbourhood $U^{**}\subset\wt U$ of the point $x^*$ from Assumption \ref{ass:3}. \\
1) We show (\ref{nummelin}) as a mere inequality. By a) and b) of Corollary \ref{cor:4bis},
$$ \int_{U^{**} } r(x^*, y ) dy > 0 .$$ 
As a consequence, there exists $y^* \in U^{**} $ such that $ r (x^* , y^* ) > 0.$ By lower semicontinuity, this can be extended to small balls $ B_{\vep^*} (x^* ) $ and $ B_{\vep^* } ( y^* ) , $ where we can assume w.l.o.g. that $ B_{\vep^* } ( y^* )  \subset U^{**} .$\\ 
2) It remains to check that $C^*:=B_{\vep^*} (x^*)$ is visited infinitely often by the sampled chain: then  $C^*$ will be `small' in the sense of Nummelin \cite{Num-85}, so (\ref{nummelin}) will be Nummelin's minorization condition (M1) for the sampled chain. 
Let $K\subset E$ denote the compact appearing in the Lyapunov condition of Assumption \ref{ass:2}~b). 
Then for arbitrary choice of a starting point $x\in E$, the skeleton chain $(X_{kT})_k$ visits $K$ infinitely often, almost surely. The sampled chain $(Z_\ell)_\ell$, being `sampled' out of the skeleton chain by tossing independent coins at times $kT$ with success probability $1{-}p$,  
inherits this property.  
Thanks to Corollary \ref{cor:4bis} we know that $ x \to R( x, C^*) $ is lower semicontinuous and positive for any fixed $x.$ As a consequence, since a lower semicontinuous function has a minimum on each compact set, we obtain
$$ \inf_{ x \in K} R(x, C^* ) > 0 .$$  Then Borel-Cantelli implies that the sampled chain $ (Z_\ell)_{ \ell \in \bbn_0}$ visits $C^* $ infinitely often, almost surely, for arbitrary choice of a starting point $x\in E$. {\hfill $\bullet$ \vspace{0.25cm}}

\begin{proof}{\bf of Theorem \ref{theo:1}} \quad 
It is sufficient to prove a), or equivalently, to prove positive Harris recurrence of the sampled chain $ (Z_\ell)_{ \ell \in \bbn_0}$,  see \cite{Rev-84}. Having proved Lemma \ref{lem:1.5} for the sampled chain, the Nummelin splitting condition (\ref{nummelin}) implies Harris recurrence of the sampled chain (Nummelin \cite{Num-78}, \cite{Num-85}). 
It remains to show that recurrence is necessarily positive recurrence. The sampled chain being a Harris recurrent $T-$chain
by Corollary \ref{cor:4bis}, Theorem 3.2 a) of Meyn and Tweedie \cite{MeyTwe-92} shows that every compact set is a `petite set'. This holds in particular for the compact $K$ in the Lyapunov condition of Assumption \ref{ass:2} b). Hence Assumption \ref{ass:2} b) corresponds to condition (DD2) of \cite{MeyTwe-92}, and Theorem 4.6 of \cite{MeyTwe-92} finishes the proof.  
\end{proof}

%%%%%%%%%%%%%%%%%%%%%%%%%%%%%%%%%%%%
\subsection{Proof of Proposition \ref{cor:1bis} in the analytic case. Global Lebesgue densities.}\label{sec:analyticharris}
%%%%%%%%%%%%%%%%%%%%%%%%%%%%%%%%%%%%

In case where the coefficients of \eqref{sde} are real analytic functions 
we obtain stronger results.  

\begin{lem}\label{lem:2bis}
Under Assumptions \ref{ass:1}--\ref{ass:3} and \ref{ass:5} the following holds true.\\
a) For $0<t<\infty$, transition probabilities $P_{0,t}(x,dy)$ of the process (\ref{sde}) admit Lebesgue densities $p_{0,t}(x,y)$ with the following properties: $y\to p_{0,t}(x,y)$ is   
infinitely differentiable  
on ${\rm int}(E)$ when $x\in E$ is fixed, and $x\to p_{0,t}(x,y)$ is lower semicontinuous on $E$ when $y\in {\rm int}(E)$ is fixed. \\
b) $P_{0, t } $ is a strong Feller transition semigroup, i.e.\ $ P_{0, t } f \in {\cal C}_0 $ if $ f $ is a bounded $\cale -$measurable function. 
\end{lem}

\vskip-0.5cm
\begin{proof}
By Theorem \ref{theo:6}, the weak Hoermander condition holds on the full state space $E$, thus $\wt U = {\rm int}(E)$ in Lemma \ref{lem:1.5} and in Lemma \ref{lem:2}. In particular, for every measurable set $ A\subset \wt U$  and for all $t$, the mapping $x \to P_{0,t}(x,A)$ is lower semicontinuous on $E$. The boundary $\partial E \cap E$ being entrance only, by Assumption 1 b), the last assertion extends to all $A\in\cale$. As a consequence, for any positive and bounded $\cale -$measurable function $ f ,$ the mapping $ x \to  P_{0,t}f (x) $ is lower semicontinuous on $E$. We now use the following argument of Ichihara and Kunita \cite{Ich-Ku}, proof of lemma 5.1, to deduce the strong Feller property of the semigroup. Putting $\tilde f := \|f\|_{\infty } - f , $ we conclude that $  P_{0,t}\tilde f $ and hence $- P_{0,t}f $ are lower semicontinuous. As a consequence, $P_{0,t}f $ must be continuous. 
\end{proof}\\

\begin{proof}{\bf of Proposition \ref{cor:1bis} }  
By Lemma \ref{lem:2bis}, $\mu=\mu P_{0,T}$ and all $\mu P_{0,s}$ admit Lebesgue densities.  
\end{proof}

%%%%%%%%%%%%%%%%%%%%%%%%%%%%%%%%%
\section{Proofs for Sections \ref{sec:HH} and \ref{sec:limit theorems}}\label{sec:5}
%Proofs of Propositions \ref{prop:1} and \ref{prop:2}}}
%%%%%%%%%%%%%%%%%%%%%%%%%%%%%%%%%%%%

We consider CIR-HH and OU-HH with all notations of Section \ref{sec:HH}, and give proofs for Propositions \ref{prop:1} and \ref{prop:2} and Theorems \ref{theo:cham} and \ref{theo:4}.

\begin{proof}{\bf of Proposition \ref{prop:1}}\quad  
1) We have $E = \bbr \times [0,1]^3 \times (0,\infty)$. Since all coefficients of CIR-HH are analytic on $U\supset E$,  Assumption \ref{ass:5} holds.  
By choice of $S(\cdot)$ in (\ref{defCIR}), Assumption \ref{ass:2}~a) is satisfied. Let us now assume that $2a>1$, select a nonnegative $\calc^2$-function $\psi$ defined on $(-\infty,\infty)$ satisfying $\psi(y):=|y|$ when $|y|>1$, and prove that $V:E\to [1,\infty)$ defined for all $x=(x^1, x^2, x^3, x^4, x^5) \in E$ by 
\beqq\label{CIRlyapunov}
V(x) \;:=\; 1 + \log^2(x^5) + (x^5)^2 + \psi(x^1),
\eeqq
is a Lyapunov function for the skeleton chain $(X_{kT})_{k\in\bbn_0}$ in the sense of Assumption \ref{ass:2}~b). First, the fifth component $(\xi_t)_{t\ge 0}$ in system (\ref{xiHH_general})+(\ref{defCIR}), taken separately, has the Markov generator $\wt\call_t$ given by
$$
\frac{\partial}{\partial t} \;+\; (a + S(t) - x^5) \frac{\partial}{\partial x^5} \;+\; x^5 \frac{\partial^2}{\partial (x^5)^2}. 
$$
Because of $2a>1$ and $S(\cdot)\ge 0$, we find positive constants $\wt c_1, \wt c_2$ such that the function $\wt V:(0,\infty)\to [1,\infty)$ defined by $\wt V(x^5):= 1 + \log^2(x^5) + (x^5)^2$ satisfies 
$$
\wt\call_t\, \wt V(x^5) \;\le\; - \wt c_1\, \wt V(x^5)  \;+\; \wt c_2 \quad\mbox{for all $t\ge 0$ and all $x^5\in (0,\infty)$}. 
$$
Second, the 5d process $(X_t)_{t\ge 0}$  has the Markov generator $\call_t$ given by
$$
\frac{\partial}{\partial t} + \sum_{i=1}^5 b^i(t,x) \frac{\partial}{\partial x_i} + \frac12 \sum_{i,j=1}^5 (\si \si^{\!\top})^{i,j}(x) \frac{\partial^2}{\partial x_i \partial x_j}   
$$
with notations of (\ref{sde}), (\ref{xiHH_general}) and (\ref{defCIR}). By definition in (\ref{function_F}) there are positive constants $d_1<d_2$ such that $F(v,n,m,h)$ takes values between $d_1 v$ and $d_2 v$ when $v\to+\infty$, and between  $d_2 v$ and $d_1 v$ when $v\to-\infty$, uniformly over $n,m,h$ in $[0,1]^3$. This results in positive constants $c_1, c_2$ such that $V:E\to [1,\infty)$ defined by (\ref{CIRlyapunov}) has the property  
\beqq\label{ausgangsbasis-lyapunov}
\call_t\, V(x) \;<\; - c_1\, V(x)  \;+\; c_2    \quad\mbox{for all $t\ge 0$ and all $x\in E$}. 
\eeqq
Localizing with $\tau_m:=\inf\{ t: |V(X_t)|>m \}$ as $m\to\infty$, we deduce from (\ref{ausgangsbasis-lyapunov}) 
$$
E_x\left( V( X_T ) \right)  \;\;\le\;\; V(x)\, e^{-c_1 T} \;+\; \int_0^T c_2\, e^{-c_1 t}\, dt 
$$
(a well-known argument, cf.\ (2.2)-(2.4) in Mattingly, Stuart and Higham \cite{MatStuHig-02}) and rewrite this as    
$$
P_{0,T} V \;\le\; \la V \,+\, \delta  \quad\mbox{on $E$},  
$$
for some constants $0<\la<1$ and $\delta>0$.  The function $V:E\to[1,\infty)$ being `bowl-shaped', this is an essentially stronger assertion than  Assumption \ref{ass:2} b).
So far, we have checked Assumptions \ref{ass:1}, \ref{ass:2} and \ref{ass:5}. 
\\
2) We now turn  
to the key Assumption \ref{ass:3}. We have to specify a point $x^*$ in ${\rm int}(E)$ of full weak Hoermander dimension and attainable in a sense of deterministic control. Our candidate in ${\rm int}(E)$ is 
$$ 
x^* = ( v^* , n^* , m^*, h^* , 1) := (v^0, n_\infty (v^0), m_\infty (v^0) , h_\infty (v^0 ) , 1) \;. 
$$
We have shown in 
\cite{HH-TD} (see Theorem 3, Proposition 6 and Section 5.4 there) 
that any point whose first four coordinates coincide with $(v^0, n_\infty (v^0), m_\infty (v^0) , h_\infty (v^0 ))$ is of full weak Hoermander dimension. 
It remains to prove that $x^*$ is attainable. Control systems $t\to\vph(t)$ make use of Stratonovich drift. For CIR-HH, Stratonovich correction (\ref{strato}) affects only the first and the fifth component:   
$$
\wt b^i(t,x) = b^i(t,x) - \frac12 \si^5(x)\frac{\partial\si^i}{\partial x^5}(x) 
\;=\; b^i(t,x) - \frac14 \;,\; i=1,5 \quad,\quad \wt b^i(t,x)=b^i(t,x)\;,\; i=2,3,4 \;. 
$$
Hence, for every choice of a starting point $x\in E$, we have to construct an $\dot {\tt h} \in L^2_{\rm loc}$ which drives the 5d trajectory $\vph=\vph^{({\tt h},x,x^*)}$ satisfying   
\beqq\label{CIR-HH-controlled}
\left\{\begin{array}{l}
\frac{d}{ds }  \vph^1(s)   \;=\;  \frac{d}{ds} \vph^5(s) \;- F( \vph^1(s), \vph^2(s), \vph^3(s), \vph^4(s) ) \\
\frac{d}{ds }  \vph^2(s)  \;=\;  \, \al_n(\vph^1(s))\,(1-\vph^2(s))  \;-\; \beta_n(\vph^1(s))\, \vph^2(s)   \\
\frac{d}{ds }  \vph^3(s)  \;=\;  \, \al_m(\vph^1(s))\,(1-\vph^3(s))  \;-\; \beta_m(\vph^1(s))\, \vph^3(s)   \\
\frac{d}{ds }  \vph^4(s)  \;=\;  \, \al_h(\vph^1(s))\,(1-\vph^4(s))  \;-\; \beta_h(\vph^1(s))\, \vph^4(s)   \\
\frac{d}{ds} \vph^5(s) \;=\; [\,a - \frac14 + S(s) - \vph^5(s)\,] \;+\; \sqrt{\vph^5(t)\,}\; \dot{\tt h}(s)  
\end{array}\right. 
\eeqq 
from $\,x= \vph(0)\,$ to $\,x^* = \lim\limits_{t\to\infty}\vph(t) \,$. For ease of notation we set $\,\wt a \;:=\; a\;-\;\frac14. $ 

We start with an intuitive argument which is not yet a rigorous one. Consider $x=(x^1, x^2, x^3, x^4, \zeta) \in E$ where $\zeta\in (0,\infty)$ is arbitrary.  
In order to push the fifth component of (\ref{CIR-HH-controlled}) with initial value $\zeta$ towards the desired limit $1$,  
select a $\calc^\infty$ function $\gamma_5^{(\zeta,1)}\,$ with the properties 
\beqq\label{requirements_control}
\left\{\begin{array}{l}
\gamma_5^{(\zeta,1)}(0) = \zeta \quad,\quad \gamma_5^{(\zeta,1)}(t) > 0 \;\;\mbox{for all $t\ge 0$ and all $\zeta>0$} \\
\gamma_5^{(\zeta,1)}(t) = 1 \quad\mbox{for all $t\ge |\zeta-1|+1$} \\
| \frac{d}{dt}\gamma_5^{(\zeta,1)}(t) | \;\le\; 1 \quad\mbox{for all $t\ge 0$ and all $\zeta>0$} 
\end{array}\right.
\eeqq 
(e.g.,  apply a $\calc^\infty$ smoothing kernel with support $[-\frac12,\frac12]$ to $t\to (\zeta-t)\vee 1$ if $\zeta>1$, and to $t\to (\zeta+t)\wedge 1$ if $0<\zeta<1$).  
Prescribing $\vph^5(t) \;:=\; \gamma_5^{(\zeta,1)}(t)$ for all $t\ge 0$, the fifth equation in~(\ref{CIR-HH-controlled}) 
$$
\frac{d}{dt}\vph^5(t) \;=\; [\, \wt a + S(t) - \vph^5(t)\,] \;+\; \sqrt{\vph^5(t)\,}\, \dot{\tt h}(t) 
$$
with $\wt a := a-\frac14$ determines a one-dimensional control $\,\tt h\,$ by
$$
\dot{\tt h}(t) \;:=\; \frac{ \frac{d}{dt}\vph^5(t) \,-\, [\, \wt a + S(t) - \vph^5(t)\,] }{\sqrt{\vph^5(t)\,} }
\;=\; \frac{ \frac{d}{dt}\gamma_5^{(\zeta,1)}(t) \,-\, [\, \wt a + S(t) - \gamma_5^{(\zeta,1)}(t)\,] }{\sqrt{\gamma_5^{(\zeta,1)}(t)\,} }, \quad\quad t\ge 0. 
$$
We have $\,\dot{\tt h} \in L^2_{\rm loc}$ by  (\ref{requirements_control}). With this choice for $\,\dot{\tt h}\,$, 
the coordinates $(\vph^1,\vph^2,\vph^3,\vph^4)$ of the control system (\ref{CIR-HH-controlled}) form a deterministic Hodgkin-Huxley system (\ref{HH_det}) where constant input $\,c\, $ is replaced by $\frac{d}{dt}\gamma_5^{(\zeta,1)}(t)$.  
Since $\frac{d}{dt}\gamma_5^{(\zeta,1)}(t)$ is in absolute value $\le 1$ and vanishes for $t > |\zeta-1|+1$, this deterministic Hodgkin-Huxley system will make before and after time $|\zeta-1|+1$ at most a finite number of spikes and then spiral into the equilibrium point $(v^0, n_\infty(v^0), m_\infty(v^0), h_\infty(v^0))$ corresponding to constant input $c=0$ which is stable. 
It is here that our argument is not rigorous: for the deterministic HH with constant input $c=0$, we are not able to prove that indeed {\em every} starting point $(v,n,m,h)\in\bbr{\times}[0,1]^3$ belongs to the basin of attraction of $(v^0, n_\infty(v^0), m_\infty(v^0), h_\infty(v^0))$. 
\\
3) In the remaining part of this subsection we construct a different control which leads to a rigorous proof. This construction, much more lengthy than the previous heuristic argument, is organized in parts I) to VI) below.\footnote{
%%%%%%%%%anfangfussnote
In fact, if we are not able to assert that every starting point $(v,n,m,h)\in\bbr{\times}[0,1]^3$ belongs to the deterministic basin of attraction of $(v^0, n_\infty(v^0), m_\infty(v^0), h_\infty(v^0))$, for the deterministic HH with constant input $c=0$, our control below together with the support theorem amounts to prove that arbitrarily small neighbourhoods of the point $x^*=(v^0, n_\infty(v^0), m_\infty(v^0), h_\infty(v^0),1)$ in $E$ can be attained with positive probability from every starting point $(v,n,m,h,\zeta)\in E$, at least after some long amount of time, in the CIR-HH process $X$ defined by (\ref{xiHH_general})+(\ref{defCIR}).
%%%%%%%%%%%%%endefussnote
}
Before going into the details, let us briefly describe the main points of our construction which proceeds roughly in three steps. 
The first step uses properties of the deterministic HH with constant input $c = 0$ in order to force the first coordinate $ \vph^1 $ in (\ref{CIR-HH-controlled}) into a "good" interval from which it will never escape again. It turns out that $ (- 12, 120)$ is such an interval. In a second step, 
we force the fifth component $ \vph^5$ to attain sufficiently high values guaranteeing that it will never touch $0$ during our construction. In this way, we ensure that $ \vph^1 $ and $\vph^5$ take values in a "good" subset of the state space. At this point we start the main part of our construction, the third step, 
in which we first force $ \vph^1$ into its limit value $v^*,$ during a fixed time period, and then use the fact that for fixed $ \vph^1$, the gating variables $ \vph^i, \, i\in\{2,3,4\}$, converge exponentially fast to their respective equilibria. We now give the details of the construction.
\\
Part I) We collect some auxiliary facts. For $F(v, n , m , h)$ defined in (\ref{function_F}) we have 
\begin{equation}\label{def-f}
\sup_{ -12 < v < 120 \,,\, 0\le n,m,h\le 1 } |F (v, n , m , h)| \;\le\; f  
\end{equation}
for some constant $f$, and (\ref{function_F}) allows for bounds 
\beqq\label{bounds_F_weitdraussen}
F(v,0,0,0) = \inf_{n,m,h}F(v,n,m,h) \;\mbox{if $v\ge 120$} \;, \; F(v,0,0,0) = -\inf_{n,m,h}|F(v,n,m,h)| \;\mbox{if $v\le -12$} . 
\eeqq
Fix $v \in (-12, 120) $, let $m, n, h \in [0,1]$ be arbitrary, let $t \to (\bar n_t (v) ,\bar m_t (v), \bar h_t(v) )$ denote the solution to 
\begin{equation}\label{eq:nmh}
\left\{\begin{array}{l}
\frac{d}{ds }  \bar n_s (v)  \;=\;  \, \al_n(v)\,(1-\bar n_s(v))  \;-\; \beta_n(v)\, \bar n_s(v)   \\ 
\frac{d}{ds }  \bar m_s (v)  \;=\;  \, \al_m(v)\,(1-\bar m_s (v))  \;-\; \beta_m(v)\, \bar m_s  (v)  \\
\frac{d}{ds }  \bar h_s (v) \;=\; \, \al_h(v)\,(1-\bar h_s (v) )  \;-\; \beta_h(v)\, \bar h_s  (v) 
\end{array}\right.
\end{equation}
with initial value $(n,m,h)$ at $t=0$. Whenever there is no ambiguity about the value $v$ which we keep constant we shall write $(\bar n_t,\bar m_t,\bar h_t)$ for short. We see from (\ref{expl_expr_gating_var})+(\ref{nmhinfini}) and (\ref{function_F}) that there are positive constants $C$ and $ \lambda $ such that uniformly in $ (v,n,m,h) \in (-12, 120 )\times [0,1]^3$   
\begin{equation}\label{eq:lambda}
\max\left\{\, |\bar n_{s} - n_\infty(v)| \,,\, |\bar m_{s} - m_\infty(v)|  \,,\, |\bar h_{s} - h_\infty(v)| \, \right\} \;\le\; Ce^{-\lambda s} 
\end{equation}
\begin{equation}\label{eq:lambdabis}
\left| F\left( v, \bar n_{s}, \bar m_{s}, \bar h_{s} \right) - F ( v, n_\infty(v), m_\infty(v) , h_\infty(v)\right| \;\le\; Ce^{-\lambda s} 
\end{equation}
for all $s\ge 0$. 
\\
Part II) Fix any starting value $x=(v,n,m,h,\zeta)$ in $E$. We now show that during some first phase of our control it is possible to keep $\vph^5$ fixed equal to $\zeta$ and force $\vph^1$ into the interval $(-12,120)$. 
\\
i) Consider the deterministic HH system (\ref{HH_det}) with constant input $c=0$ which corresponds to $\frac{d}{ds} \vph^5=0$ in (\ref{CIR-HH-controlled}).  
As long as its first variable $\vph^1$ takes values outside $(-12,120)$, the term $\,-F(\vph^1,\cdot,\cdot,\cdot)$ in the first equation of (\ref{CIR-HH-controlled}) represents a back-driving force,  and the bounds (\ref{bounds_F_weitdraussen}) allow to `decouple' $\vph^1$ from the gating variables $ (\vph^2,\vph^3,\vph^4) $  in $[0,1]^3$.  
Consequently   
\beqq\label{def-t_1} 
t_1 \;:=\; \inf \{ t :   \vph^1(t) \in (-12, 120 ) \} \;<\; \infty.
\eeqq
Inspection of the fifth equation of (\ref{CIR-HH-controlled}) shows that keeping $\vph^5$ equal to $\zeta$ on $[0,t_1]$ amounts to prescribe 
$$ 
\dot{\tt h}(s) \;:=\; \frac{-[ \wt a + S(s) - \zeta]}{\sqrt{\zeta\,}} 
$$
on this interval which is meaningful since $\zeta>0$. 
\\
ii) Once arrived within $(-12, 120 )$, the deterministic HH with constant input $c=0$ will never be able to leave this interval again. Actually a stronger result was proved in \cite{End-12}, proposition 1.8: a deterministic HH with initial first component in $(-12, 120 )$ and time-dependent input $c(t)\,dt$ such that $-6.78 < c(t) < 32.82$ for all $t$ will never be able to leave this interval. This is again a consequence of  
bounds (\ref{bounds_F_weitdraussen}) at the endpoints $v=120$ and $v=-12$ of this interval.  
\\
Part III) During some second phase $]t_1,t_2]$ of our control, with $t_2$ specified in (\ref{def-t_2}) below, we force $\vph^5$ beyond some suitably large positive threshold, and we take advantage of Part II)~ii) to make sure that $\vph^1$ does not leave the interval $(-12,120)$ during this phase. %To determine the threshold, we 
Choose  $K$ large enough such that
\begin{equation}\label{eq:K}
( K -120 - 1 )( 1 + f )\;-\; \frac{C}{\lambda } \;\;>\;\; 1  
\end{equation}
with $C$, $ \lambda $ and $f$ the constants of (\ref{eq:lambda}), (\ref{eq:lambdabis}) and (\ref{def-f}). Now we prescribe 
$$ 
\frac{d}{ds }\vph^5(s) \;\equiv\; 1  \quad\mbox{or equivalently}\quad \dot{\tt h}(s) \;:=\; \frac{1 - [\wt a + S(s) - \vph^5(s)]}{\sqrt{\vph^5(s)\,}} \quad,\quad t_1< s      
$$
and stop this second phase of control at time 
\beqq\label{def-t_2} 
t_2 \;:=\; \inf \{ s :\, \vph^5(s)  \ge  K (f+1) \,\} \;<\; \infty  \;. 
\eeqq
Note that on $[t_1, t_2]$ the first four components of (\ref{CIR-HH-controlled}) coincide with the solution of a deterministic HH system with constant input $c=1$. Its first component lies in $(-12, 120 )$ at time $t_1$. Part II)~ii)  applies and ensures that $\vph^1$ remains in $(-12, 120 )$ on $[t_1, t_2]$.   
\\
Part IV)  Let us set $\vph(t_2):=(v',n',m',h',\zeta')$. By construction $v'=\vph^1(t_2)  \in (-12, 120)$ and $\zeta'= \vph^5(t_2) \ge  K (f+1) $.  
Next, during some time interval $]t_2,t_3]$ with $t_3$ to be specified in (\ref{def-t_3}) below,  
our control acts on $\vph^1$ and directs it towards the desired limit $v^*=v^0$; 
we will have to make sure that $\vph^5$ remains positive  during this phase to prevent $\vph$ from leaving the state space $E$. 
\\
i)  In analogy  to (\ref{requirements_control}), we choose  a  $\calc^\infty$-function $\,\gamma_1  =  \gamma_1^{(v',v^*)}: [t_2,\infty)\to\bbr$ satisfying 
\beqq\label{prescription-gamma1}
\left\{\begin{array}{l}
\gamma_1(t_2)\;=\; v' \\
|\frac{d \gamma_1}{ds}|(s) \;\le\; 1 \quad\mbox{on $ [t_2, t_2 + |v'-v^*| + 1)$}   \\
\gamma_1(s) \equiv v^* \;\;\mbox{on $[t_2 + |v'-v^*| +1,\infty)$} \;. 
\end{array}\right.
\eeqq 
As a consequence of (\ref{equilibrium-2}), $\,v^*=v^0$ is strictly positive, hence we have $|v'-v^*|\le 120$  for all $v'$ under consideration in this step.  
For the present phase of control we prescribe 
$$
\vph^1 (s) \;:=\; \gamma_1^{(v',v^*)}(s)   \quad\mbox{for}\quad s\ge t_2 \;. 
$$
Plugging this into the first and fifth equations in (\ref{CIR-HH-controlled}), we specify the control as  
\beqq\label{prescription-ctrl-gamma1}
\dot{\tt h} \;:=\; \frac{ \frac{\partial \gamma_1}{\partial s} + F(\vph^1,\vph^2,\vph^3,\vph^4) - [\wt a + S - \vph^5] }{\sqrt{\vph^5\,}} \quad\mbox{for}\quad s\ge t_2 \;. 
\eeqq
Clearly our choice (\ref{prescription-gamma1})+(\ref{prescription-ctrl-gamma1}) drives $\vph^1$ into the desired limit $v^*=v^0$ and fixes it there, and clearly $\vph^1$ does not leave the interval $(-12,120)$. 
\\
ii) We have to ensure that $\vph^5$ remains positive. With $\,\dot{\tt h}\,$ specified by (\ref{prescription-ctrl-gamma1}), the first equation of (\ref{CIR-HH-controlled}) combined with (\ref{prescription-gamma1})+(\ref{def-f}) shows  
$$
\frac{d}{ds }  \vph^5(s)  \;\geq\; -1 - f  \quad\mbox{for}\quad s\ge t_2.  
$$
Therefore at time $\wt s:=t_2+120+1$ by (\ref{def-t_2})
$$ 
\vph^5(\wt s) \;\geq\; \vph^5(t_2)  - (1+f)(120+1)  \;\geq\; (K-120-1)(1+ f).  
$$ 
Since $t_2+120+1$ is an upper bound for all times $t_2 + |v'-v^*| +1$ which can be considered in (\ref{prescription-gamma1}), the choice of $K$ in (\ref{eq:K}) provides us with the following intermediate result:  
\beqq\label{zwischenbilanz}
\vph^1(\wt s) \;=\; v^* = v^0 \;\;\mbox{and}\;\; \vph^5(\wt s)  \;>\; 1+\frac{C}{\lambda } \quad\mbox{at time $\wt s=t_2+120+1$} \;. 
\eeqq 
However $\wt s$ arising here is not yet the time at which we will stop the present phase of control. 
\\
iii)  On the time interval  $[\tilde s, +\infty[$,  $\,\vph^1$ remains fixed in $v^*=v^0$  by choice in (\ref{prescription-gamma1}), the gating variables $(\vph^2, \vph^3, \vph^4)$  converge exponentially fast towards  $(n_\infty(v^0), m_\infty(v^0), h_\infty(v^0)) = (n^*, m^*, h^*)$ by (\ref{eq:lambda}), whereas the first equation in (\ref{CIR-HH-controlled}) yields  
$$
\frac{d}{ds }  \vph^5(s) \;=\; \frac{d}{ds }   \vph^1(s) + F ( \vph^1_s, \ldots, \vph^4_s) \;=\; 0 \;+\; F ( \vph^1_s, \ldots, \vph^4_s)\quad\mbox{for}\quad s\ge \tilde s.  
$$
The crucial point in our present construction  (in contrast to the controls which we did consider in \cite{OU-HH}) is that $F\left(v^0, n_\infty(v^0), m_\infty(v^0), h_\infty(v^0)\right)$ vanishes, see (\ref{F_infinity})--(\ref{equilibrium-2}). This fact allows in combination with (\ref{eq:lambdabis}) for a finite integral 
\beqq\label{finitelimitphi5}
\vph^5(\infty) \;:=\; \vph^5(\tilde s) + \int_{\tilde s}^\infty F\left( v^0, \bar n_u (v^0 ), \bar m_u (v^0 ) , \bar h_u (v^0 ) \right) du    
\eeqq 
together with a bound 
$$ 
\frac{d}{ds }   \vph^5(s) \;\geq\;  - C e^{ - \lambda s} \quad\mbox{ for $\; s \geq \tilde s$}. 
$$
Thanks to (\ref{zwischenbilanz}) both last assertions imply   
for $s\ge\wt s$ 
\begin{equation}\label{eq:above}
\vph^5(s) \;\geq\;  \vph^5(\tilde s)  - C \int_{\tilde s}^s e^{- \lambda t } dt 
\;\;\geq\;\; \vph^5(\tilde s)  - \frac{C}{\lambda} %e^{- \lambda \tilde s }  
\;\;>\;\; (1+\frac{C}{\lambda})-\frac{C}{\lambda}%e^{- \lambda \tilde s }   \quad\mbox{ for $\; s \geq \tilde s$}. 
\;=\; 1 \;. 
\end{equation}  
Thus $\vph$ does not leave the state space $E$ for $ s \geq \tilde s$ and we have  $\vph^5(\infty)  \;>\; 1$. 
\\
iv) If we recapitulate the arguments above, we see that the control defined so far drives $\vph$ towards
\beqq\label{1stlimit}
\left( v^*, n^*, m^*, h^*, \vph^5(\infty)  \right) = 
\left( v^0, n_\infty(v^0), m_\infty(v^0), h_\infty(v^0), \vph^5(\infty) \right) \quad\mbox{where}\quad \vph^5(\infty) \;>\; 1 
\eeqq
as $t\to \infty$, whose first four components coincide with those of the desired limit $x^*=\left( v^*, n^*, m^*, h^*, 1  \right)$ defined in (\ref{CIRcandidate}) above. However our argument is not yet finished, for two reasons. First $\vph^5(\infty)$ may be far from $1$ which is the fifth component of $x^*$. Second the integral contributing to (\ref{finitelimitphi5}) keeps trace of the initial conditions, in contrast to what was required in (\ref{CIRcandidate}) --or in Definition \ref{def:1}-- for the limit point $x^*$ for $\vph$ as $t\to\infty$. 
Therefore we stop the control constructed above at time 
\beqq\label{def-t_3}
t_3 \;:=\; \inf\left\{ s > \tilde s: \vph^2(s)\in B_\vep(n^*) \,,\, \vph^3(s)\in B_\vep(m^*) \,,\,\vph^4(s)\in B_\vep(h^*)  \right\}  
\eeqq 
where $\vep>0$ is arbitrarily small but fixed. It remains to arrange for the fifth component and to get rid of last traces of the initial conditions. Introduce the notation  
$$
\vph(t_3) \;=:\; ( v^0 ,n'',m'',h'',\zeta'')  \quad\mbox{where}\quad  \zeta'' > 1 \;,\; n''\in B_\vep(n^*) \;,\; m''\in B_\vep(m^*)\;,\; h''\in B_\vep(h^*) \;. 
$$
Part V) During some time interval $]t_3,t_4]$ with $\,t_4\,$ to be defined in (\ref{def-t_4}) below, we drive $\vph^5$ towards $v^*=1$ without altering much $\vph^1, \ldots, \vph^4$.  For this, we prescribe the first component of $\vph$ as 
$$ 
\vph^1(s) \;:=\; \gamma_2(s) \quad\mbox{for $\;s\in [t_3,\infty)\;$} 
$$
where $\,\gamma_2\,$ is a $\calc^\infty$ function supported by $[t_3,\infty)$ with the properties  
\beqq\label{prescription-gamma2}
\left\{ \begin{array}{l}
\gamma_2(t_3) \;=\; v^0   \\
- 10^{-k} \leq \frac{\partial \gamma_2}{\partial s}(s) \leq 0  \quad \mbox{for  $t_3\le s\le t_3+1$} \\
\gamma_2(s) \;=\; v^0 - 10^{-k} \quad \mbox{for $s\ge t_3+1$}   \\
\end{array}\right.
\eeqq
for some $k\in\bbn$ large enough to ensure that $\vph(t_3+1) \;=:\; (\, v^0 - 10^{-k}\,,\, n{'''}\,,\, m{'''}\,,\, h{'''}\,,\, \zeta{'''} \,)$ satisfies 
\beqq\label{sideeffect-gamma2}
 n{'''}\in B_{2\vep}(n^*) \;,\;  m{'''}\in B_{2\vep}(m^*) \;,\; h{'''}\in B_{2\vep}(h^*)\;,\;\zeta{'''} > 1 \;. 
\eeqq 
Let us set $v^{**}:=v^0 - 10^{-k}$. Using notation (\ref{eq:nmh}) with $v$ fixed to $v^{**}$, define $\bar n_s:=\bar n_s(v^{**})$, $\bar m_s:=\bar m_s(v^{**})$, $\bar h_s:=\bar h_s(v^{**})$ for $s\ge t_3+1$, with starting value $(n{'''}\,,\, m{'''}\,,\, h{'''})$ at $(t_3+1)$. According to (\ref{eq:lambda})+(\ref{eq:lambdabis}), these converge exponentially fast to  $n_\infty(v^{**}), m_\infty(v^{**}), h_\infty(v^{**})$, 
and again we have 
$$
\frac{d}{ds }  \vph^5(s) \;=\; \frac{d}{ds }   \vph^1(s) + F ( \vph^1_s, \ldots, \vph^4_s) \;=\;  0 \;+\; F(v^{**}, \bar n_s, \bar m_s, \bar h_s)  \;\lra\; F_\infty(v^{**}) \quad\mbox{as $s\to\infty$} \;, 
$$
$$
\vph^5(t_3+1+\tau) \;:=\;\zeta{'''} + 0 + \int_{t_3+1}^{t_3+1+\tau} F ( v^{**}, \bar n_u, \bar m_u, \bar h_u ) du  \quad\mbox{for finite time $0<\tau<\infty$} \;. 
$$
$F_\infty$ being strictly increasing and $F_\infty(v^0)=0$, see (\ref{function_F})--(\ref{F_infinity}), we have $F_\infty(v^{**})<0$, hence the last integral as a function of $\tau$ is eventually decreasing (with extremely small slope by suitable choice of~$k$). 
During this phase of control, $(\vph^1(s), \vph^2(s), \vph^3(s), \vph^4(s) )$ remains close to its value at time $t_3$ by (\ref{prescription-gamma2})-(\ref{sideeffect-gamma2}) and the behavior of $\bar n_s , \bar m_s , \bar h_s$ above. We stop this phase of control at time $t_4$ defined by 
\beqq\label{def-t_4}
t_4 \;:=\; \inf\{\, t > t_3 \,:\, \vph^5(t) = 1 \,\}.
\eeqq
VI) What we have achieved so far, putting steps I)--V) together, is a control $\vph$ which at time $t_4$ attains 
\beqq\label{finalapprox}
\vph(t_4) \;=\; \left(\, v^*\pm\wt\vep \,,\, n_\infty(v^*)\pm\wt\vep\,,\, m_\infty(v^*)\pm\wt\vep\,,\, h_\infty(v^*)\pm\wt\vep\,,\,  1\,\right)  
\eeqq
for some $\wt\vep>0$ depending on $\vep>0$ and $k\in\bbn$ chosen above. Note that we can achieve (\ref{finalapprox}) for $\wt\vep>0$ arbitrarily small. From now on, the intuitive argument of step 2) above will be a rigorous one: we know that the equilibrium point (\ref{equilibrium-2}) is stable, and start an ultimate phase of control at time $t_4$ in some point (\ref{finalapprox}) arbitrarily close to the equilibrium (\ref{equilibrium-2}). Then with $\vph^5\equiv 1$ on $[t_4,\infty)$  and a control $\,\dot{\tt h}\,$ determined from $\frac{\partial \vph^5}{\partial s}=0$, a 4d deterministic Hodgkin-Huxley 
system with constant input $c=0$ indeed spirals into the desired equilibrium value (\ref{equilibrium-2}), and we are done. 
\end{proof}

\begin{proof}{\bf of Proposition \ref{prop:2}}\quad For OU-HH with state space $E=\bbr\times[0,1]^3\times\bbr$ defined by (\ref{xiHH_general})+(\ref{defOU}), we consider control systems (\ref{CIR-HH-controlled}) with  fifth equation replaced by 
$\,\frac{d}{ds} \vph^5(s) = [S(s) - \vph^5(s)]  + \dot{\tt h}(s) \,.$ The proof proceeds in analogy to the preceding one, essentially simplified since there is no need to  divide by $\sqrt{\vph^5}$ or to prevent $\vph^5$ from leaving the state space. See \cite{OU-HH} for the Lyapunov function. 
\end{proof}

In order to prove Theorems \ref{theo:cham} and \ref{theo:4}, we go back to the notations of Lemma  \ref{lem:1.5} in Section \ref{sec:point}.

\begin{proof}{\bf of Theorem \ref{theo:cham} } \; 1) Consider $C^* = B_{\vep^* } (x^* )$ and $\,D^* = B_{\vep^* } ( y^* ) \subset U^{**}$ as in Lemma  \ref{lem:1.5}: $\;U^{**}$ is an arbitrarily small neighbourhood of $x^*$, and $\vep^*$ can be chosen arbitrarily small. For CIR-HH and OU-HH and any choice of a starting point $x\in E$, we know from the proofs of Lemma \ref{lem:1.5} and Theorem~\ref{theo:1} in Section \ref{sec:harris} that the skeleton chain $(X_{kT})_{k\in\bbn_0}$ visits the sets $C^*$ and $D^*$ infinitely often.\\
2) Recall %Theorem 5 of \cite{HH-TD}, i.e.\ 
the "chameleon property" of the stochastic HH which is able to mimick with positive probability on fixed time intervals the behaviour of deterministic HH's with quite arbitrary smooth deterministic input $t\to I(t)$. For $ x' = (v', n', m', h', \zeta') \in D^*$, let $\mathbb{Y}_s^{x'} $ denote the deterministic HH with input $I(\cdot)$ starting from $ (v', n', m', h') $ at time $0$, and write 
$\,\mathbb{X}^{x'}_s = (\,\mathbb{Y}^{x'}_s , \zeta' + \int_0^s I(v)dv\,)$. If the system $\,\mathbb{X}^{x'}$ does not leave $E$ during $[0,t]$, then  $\,\{ \,f \in C([0,t],E) : \sup\limits_{ 0\le s \le t } | f(s) -  \mathbb{X}^{x'}_s | \le \vep \,\}\,$ has strictly positive $Q_{x'}$-probability for every $\vep>0$, by Theorem 5 of \cite{HH-TD}. \\
3) Starting from $x'\in D^*$, assertion b) follows from 2) since the first four components of $x^*$ coincide with the equilibrium point for the deterministic HH with constant input $I(\cdot)\equiv 0$ which is attracted to this equilibrium. 
Assertion a) follows from 2) if we choose $t = k_0 T$ with $k_0\in\bbn$  sufficiently large but fixed and $I(\cdot)\equiv c$ sufficiently large such that the deterministic HH $ \,\mathbb{Y}^{x'}$ with constant input $c$ will enter the spiking regime within a finite time (smaller than $k_0T$).    
\end{proof}

\begin{proof}{\bf of Theorem \ref{theo:4} }\quad  
Recall short notation $x$ for $(0,x)$ as starting point for $\ov X$. If we can show 
\beqq\label{to_be_shown}
\mbox{for every  $0<t<\infty$ fixed} \;: \quad 
\wh F_n(t) \;\lra\; F(t) \quad\mbox{$Q_x$-almost surely as $\nto$} \;, 
\eeqq
the remaining uniformity in $t\in\bbr$ follows as in the classical Glivenko-Cantelli theorem for i.i.d.\ random variables (\cite{Bre-84}, \cite{ChoTei-88}). Recall that successive interspike times have no reason to be independent.  
\\
1) For CIR-HH and for OU-HH, all coefficients of the process being real analytic, we can do Nummelin splitting directly in the skeleton chain: by Lemma \ref{lem:2bis} we have continuous Lebesgue densities on ${\rm int}(E)$ and can find some $y^* \in {\rm int} ( E)$ with  $p_{0, T } ( x^* , y ^* ) > 0$, hence there is  $\vep^* > 0 $ such that 
$$
P_{0,T}(x', dy') \;\ge\; \al^*\; 1_{C^*}(x')\; \nu^*(dy') \quad,\quad x',y' \in E   
$$
holds with $C^* = B_{\vep^* } ( x^* ) $, $\,\nu^*$ the uniform law on $B_{\vep^*} (y^*)$, and some $\al^*>0$. 
This improves on Lemma \ref{lem:1.5} above (and on \cite{OU-HH}) and yields a sequence $(R_n)_n$ of stopping times increasing to $\infty$ such that  $R_n+1$ are renewal times for the skeleton chain $\ell \to X_{\ell T}$ which starts afresh from law $\nu^*$ at every time $R_n{+}1$. Count interspike intervals of length $\le t$ starting within the stochastic interval $[ R_jT\,,\, R_{j+1} T[$ by  
\beqq\label{Z-def-3}
Z_j(t) \;:=\; \sum_{\ell=1}^\infty 1_{ \left\{ R_j T \;\le\; \tau_\ell \;<  R_{j+1} T \right\} } \; 1_{[0, t ]} (\tau_{\ell+1}-\tau_\ell) \;.  
\eeqq
Interspike intervals having at least length $\delta>0$ by definition in (\ref{def-tau_n}),    
$$
Z_j(\infty) \;=\; \sum_{\ell=1}^\infty 1_{ \left\{ R_j T \;\le\; \tau_\ell \;<  R_{j+1} T \right\} } \;\le\;O(\frac{1}{\delta}) (R_{j+1}  - R_j) T
$$
has finite expectation under $Q_{\ov\mu}$, by positive Harris recurrence with invariant measure $\ov\mu$. 
\\
2) Keeping $t $ in (\ref{Z-def-3}) fixed, we choose $k \in \bbn $ such that $t \le (k-2) T .$ Then $Z_j( t) $ is  measurable w.r.t. ${\cal F}_{(R_{j+1}T+(k-2)T)}$ and thus, the distance between successive renewal times being at least $1$, measurable w.r.t. $\,{\cal F}_{R_{j+k-1}T}\,$. As a consequence,  
\begin{equation}\label{derniere_astuce}
\mbox{
$Z_j( t) $  is ${\cal F}_{R_{j+k-1}T}-$ measurable and independent of $(X_{(R_{j+k-1}+1)T+s})_{ s \geq 0} $  \;,
}\end{equation}
since $R_{j+k-1}{+}1$ in (\ref{derniere_astuce}) is a renewal time for the skeleton chain. Exploiting again $\,R_{j+k-1}{+}1 \le R_{j+k}\,$, (\ref{Z-def-3}) and (\ref{derniere_astuce}) imply that the rv's $\,(Z_{j + \ell k })_{\ell \geq 1}$ are i.i.d.\ for every fixed value of $j$. Thus   
$$ 
\lim_{n \to \infty } \frac{1}{nk} \sum_{m=1}^{nk} Z_m(t) \;=\; \frac{1}{k} \sum_{j=1}^{k} 
\left( \lim_{n \to \infty } \frac1n \sum_{\ell = 0 }^{n-1}  Z_{j + \ell k } \right) \;=\; C\, E_{\ov\mu} (Z_1 (t) )  
$$ 
by the classical strong law of large numbers,   
with $C$ some norming constant from choice of life cycles. 
\\
3) $\,Z_j(\infty)$ being $\calf_{R_{j+1}T}$-measurable and independent of $(\calf_{(R_{j+1}+1)T+s})_{s\ge 0}$ whereas $\,R_{j+1}{+}1 \le R_{j+2}\,$, both $(Z_{2\ell}(\infty))_\ell$ and $(Z_{1+2\ell}(\infty))_\ell$ are by definition families of i.i.d\ r.v.'s. In analogy to 2) we obtain   
$
\,\lim\limits_{\nu \to \infty } \frac{1}{\nu} \sum_{j=1}^{\nu} Z_j(\infty) = C\, E_{\ov\mu} (Z_1 (\infty) )  \,. 
$
Combining 2) and 3), the limit of the ratios 
$$
\lim_{\nu \to \infty } \frac{ \frac{1}{\nu} \sum_{j=1}^{\nu} Z_j(t) }{ \frac{1}{\nu} \sum_{j=1}^{\nu} Z_j(\infty) } \;=\; \frac{ E_{\ov\mu} (Z_1 (t) ) }{ E_{\ov\mu} (Z_1 (\infty) ) } \;=:\; F(t)
$$
exists $Q_x$-almost surely for every $x\in E$, and Theorem \ref{theo:4} is proved. 
\end{proof}

%%%%%%%%%%%%%%%%%%%%%%%%%%%%%%%%%%%%%%%%
%\appendix
\section{Appendix: Proof of Theorem \ref{theo:4bis}} \label{sec:app}
%%%%%%%%%%%%%%%%%%%%%%%%%%%%%%%%%%%
%
%   neufassung RH februar 16 
%   kommentare von Eva vom 21.02.16 19:23 eingearbeitet
%
%%%%%%%%%%%%%%%%%%%%%%%%%%%%%%%%%%%

In this appendix section we give a (somewhat technical and lengthy) proof for Theorem \ref{theo:4bis}. Notations are those established before or during the formulation of Theorem \ref{theo:4bis}; periodicity of the drift in the time variable is not really needed. The only assumptions on the process $X$ and its state space $E$ are those stated in Assumption \ref{ass:1}. In part a) of this assumption, we considered a strictly increasing sequence $(G_m)_m$ of bounded convex open sets in $\bbr^d$ such that, with $C_m := {\rm cl}(G_m),$ $\,E = \mathop{\bigcup}\limits_m C_m.$  According to part b) and b'), $\partial E \cap E$ is entrance, and from boundary points $x\in C_m\setminus G_{m+1}$, almost surely, the process $X$ immediately enters $G_{m+1}$. 
By part~c), stopping times $T_m= \inf\{ t > 0 : X_t\notin C_m\}$ are such that $ T_m \uparrow \infty $ almost surely, for every choice of a starting point $x \in E. $

{\bf Proof of Theorem \ref{theo:4bis} b) } 
We proceed by localization which works as follows. \\
1) By part~d) of Assumption \ref{ass:1}, the coefficients of SDE (\ref{sde}) are defined and $\calc^\infty$ on $[0,\infty){\times}U$, for some open $U$ containing $E=\bigcup_n C_n$. %, $C_n={\rm cl}(G_n)$. 
This allows to construct a second increasing sequence $(U_n)_n$ of bounded open convex sets such that $C_n\subset U_n$ and ${\rm cl}(U_n)\subset U$ for all $n\in\bbn$, and coefficients $(t,x)\to b_n(t,x)$ and $x\to \si_n(x)$ defined for $t\ge 0$ and $x\in\bbr^d$ and satisfying  
$$
\left\{ \begin{array}{l}
\mbox{$b_n(\cdot,\cdot)$ is bounded and globally Lipschitz on $[0,\infty){\times}\bbr^d$} \;;\\ 
\mbox{$\si_n(\cdot)$ is $\,\calc^2$ on $\bbr^d$, bounded and with bounded derivatives up to order $2$} \; ; \\
\mbox{$b_n(t,x)=b(t,x)\,$ on $[0,\infty){\times}U_{n+1}$} \;;\\
\mbox{$\si_n(x)=\si(x)\,$ on $U_{n+1}$}  \;. 
\end{array}\right.
$$
At stage $n$ of the localization, we have deterministic controls $\vph_n = \vph_n^{({\tt h}, x)}:[0,\infty)\to\bbr^d$ solving
\beqq\label{controlsystem-n}
d \vph_n (t) \;=\; \wt b_n ( t, \vph_n (t) )\, dt \;+\; \si_n( \vph_n (t) )\, \dot{\tt h}(t)\, dt  \;\;,\;\; t\ge 0 \;\;,\;\; \vph_n(0)=x
\eeqq
for all $x\in\bbr^d$ and $\,{\tt h}:[0,\infty)\to\bbr^m\,$ having absolutely continuous components $ {\tt h}^{\ell} (t) = \int_0^t \dot{\tt h }^{\ell} (s) ds $ such that $\,\dot{\tt h}^{\ell} \in L^2_{\rm loc}$, with $\wt b_n(\cdot,\cdot)$ the Stratonovich form of $b_n(\cdot,\cdot)$. 
We have a unique strong solution $X_n$ to 
\beqq\label{sde-n}
dX_n(t) \;=\; b_n(t,X_n(t))\, dt\;+\; \si_n(X_n(t))\, d W_t   \quad,\quad   t\ge 0  
\eeqq
driven by the Brownian motion $W$ of equation (\ref{sde}). We suppose that $X_n$ is defined on the same measurable space $(\Omega, {\cal A})$ as $X,$ for every choice of a starting point $x\in\bbr^d$; and we write $P_x$ for the unique probability measure on $(\Omega , \cala)$ under which all $X_n$ start from $X_n ( 0 ) = x .$ 
For arbitrary $0<t_0<\infty$ and $x\in\bbr^d$, \cite{MilSan-94} gives a control theorem for $X_n$: the support of the law of $X_n$ up to time $t_0$ starting from $x$ is the closure in $C([0,t_0],\bbr^d)$ of the set 
$$
\wt A^n_x \;:=\; \left\{  \left( \vph_n^{({\tt h}, x)}(t) \right)_{0\le t\le t_0} : \,{\tt h}\in{\tt H}\;\;\mbox{admissible}\; \right\} \;. 
$$

\vskip0.3cm
2) In the following steps 2), 3) and 4), we consider $x\in C_n$, $m,l\geq n$, and $\wt m>m$. 

We introduce stopping times $T^m_l:=\inf\{ t : X_m(t)\notin C_l\}$ and `stop' control paths (\ref{controlsystem-n}) at deterministic times $t^m_l = t^m_l({\tt h}, x) := \inf\{ t : \vph_m^{({\tt h},x)}(t)\notin C_l\}$. 
Note that a process $X_m$ may leave $E$ in finite time: hence $\sup_l T^m_l$ as well as $\sup_l t^m_l({\tt h}, x)$ may be finite. 

Recall the stopping times $T_l=\inf\{t :X_t\notin C_l\}$ from Assumption \ref{ass:1}. 
Necessarily $X_{T_m}\in C_m$ on $\{T_m<\infty\}$ since $C_m$ is compact. Then either $X_{T_m}\in G_{m+1}$, or $X_{T_m}\in C_m\setminus G_{m+1}$ in which case $(X_{T_m+s})_{s\ge 0}$ has to enter $G_{m+1}$ immediately, cf.\ Assumption \ref{ass:1}. In both cases, $(X_{T_m+s})_{s\ge 0}$ will  spend some positive amount of time in the open set $G_{m+1}$. Hence $T_m<T_{m+1}$ in both cases. As a consequence, we have $T_l<T_{l+1}$  on $\{T_l<\infty\}$ $Q_x$-almost surely for all $ l \geq n$,  together with  $T_l\uparrow\infty$ as $l\to\infty$ $Q_x$-almost surely by Assumption \ref{ass:1}. 

Paths of $X_m$ and $X$ starting from $x$ coincide up to (and even slightly beyond) time $T_{m+1}=T^m_{m+1}$. Paths of $X_{\wt m}$ and $X_m$ starting from $x$ coincide up to time $T_{m+1}$, and controls $\vph_{\wt m}^{({\tt h},x)}$ and $\vph_m^{({\tt h},x)}$ coincide up to time $t^{\wt m}_{m+1}({\tt h}, x)=t^m_{m+1}({\tt h}, x)$. 

Notice also that $t^{n- 1}_{n}({\tt h}, x)=t^n_{n}({\tt h}, x)$ since $x \in C_n. $

3) Control paths $\,\vph_m^{({\tt h}, x)}$ for $X_m$ starting at points $x\in C_n{\setminus}G_{n+1}$ remain in $C_{n+1}$ during some time interval $[0,\wt t_0]$ of positive length $\wt t_0>0$. This is seen as follows: by \cite{MilSan-94} or step 1), paths of $X_m$ can be approximated by control paths $\vph_m^{({\tt h}, x)}$ and vice versa uniformly in time $0\le t\le t_0$, for every $t_0$ fixed;  $\, X_m$ coincides with $X$ up to time $T_{m+1}\geq T_{n+1}>0$, and paths of $X$ starting in $C_n\setminus G_{n+1}$ immediately enter $G_{n+1}$ by Assumption 1.

4) When $x \in C_q$ and $t^q_{q+1}({\tt h}, x)<\infty$, we have  $\,t^m_{m+1}({\tt h}, x)>t^q_{q+1}({\tt h}, x)\,$ for all $m>q$. To see this, put $n:=q+1$ and assume $t_n^m({\tt h}, x) <\infty$ (otherwise there is nothing to prove). 
A control path $\vph_m^{({\tt h},\cdot)}$ leaving $C_n$ at time $t^m_n$ has to visit points of $\partial(C_n)$ which either belong to $C_n\setminus G_{n+1}$ or to $G_{n+1}$; in the first case, $\,\vph_m^{({\tt h},x)}$ remains in $C_{n+1}$ for some positive amount of time, by step 3). Thus in both cases, $\,\vph_m^{({\tt h},x)}$ needs a positive amount of time to leave $C_{n+1}$;  this happens at time $t^m_{n+1}\le \infty$. Thus $t^m_{n+1} ( {\tt h }, x ) $ is strictly greater than $t^m_n( {\tt h }, x )=t^{n-1}_n( {\tt h }, x )$.

5) For $x\in E$ and $\dot {\tt h}\in L^2_{\rm loc}$, define $t\to \vph^{({\tt h},x)}(t)$ by pasting together controls $\vph_l^{({\tt h},x)}$ defined so far: 
\beqq\label{pastingtogether}
\vph^{({\tt h},x)}(t)  \;:=\; 1_{\left] 0 \,,\, t^1_2({\tt h},x)\right] }(t)\; \vph_1^{({\tt h},x)}(t) \;+\; 
\sum_{l\ge 2}  1_{\left] t^l_l({\tt h},x) \,,\, t^l_{l+1}({\tt h},x)\right] }(t)\; \vph_l^{({\tt h},x)}(t)  \quad,\quad \vph^{({\tt h},x)}(0)=x \;. 
\eeqq
Notice that for $x \in C_n \setminus C_{n-1}, $ the above sum starts with $l = n- 1 , $ for which $ t^{n-1}_{n-1} ({\tt h},x) =0$ and $ t^{n-1}_{n} ({\tt h},x) \geq 0.$ Recall also $t^l_l  ({\tt h},x)=t^{l-1}_l  ({\tt h},x)$ for all $ l \geq n$, from step 2). Hence the above construction can  be achieved on a time interval 
\beqq\label{maximaltimeinterval}
0 \;\le\; t \;<\;  s({\tt h},x) \;:=\; \sup \limits_l t^l_{l+1}({\tt h},x) \;\le\; \infty \;. 
\eeqq
For $m \geq n$ and $x\in C_n$, $\,t \to \vph^{({\tt h},x)}(t)\,$ defined in (\ref{pastingtogether}) coincides with $\,t \to \vph_m^{({\tt h},x)}(t)\,$ on $\,[0,t^m_{m+1}({\tt h},x)]\,$ by step 2). Thus $\vph^{({\tt h},x)}$ solves equation (\ref{controlsystem})
$$
d \vph (t) \;=\; \wt b ( t, \vph (t) )\, dt \;+\; \si( \vph (t) )\, \dot{\tt h}(t)\, dt   \quad,\quad    \vph(0)=x
$$
up to time  $t^m_{m+1}({\tt h},x)$.  There is no reason that $s({\tt h},x)$ in (\ref{maximaltimeinterval}) should be infinite for all $x\in E$, $\dot {\tt h}\in L^2_{\rm loc}$. 
%a mapping $t\to \vph^{({\tt h},x)}(t)$ might `explode' in the state space $E=\bigcup\limits_n C_n$ by leaving successively  compacts $C_{n+1}$ at times $t^n_{n+1} ({\tt h},x)\uparrow s({\tt h},x)<\infty$. 

6) For every $\om$ fixed and every starting point $x\in E$ we can find some $\ell$ such that $X(t)\in G_\ell$ for $0<t\le 2t_0$, with $\ell$ depending on $\om$ and $x$. When $x \in {\rm int}(E)$, we have $X(t)\in {\rm int}(E)$ for $0\le t\le 2t_0$, $\,\partial(E)\cap E$ being entrance by Assumption \ref{ass:1}~b), and the assertion follows from ${\rm int}(E)=\bigcup_\ell G_\ell$, continuity of the path and compactness of $[0,2t_0]$. When $x\in \partial(E)\cap E$, we have $x\in C_{\ell_1}\setminus G_{\ell_1+1}$ for some $\ell_1$, thus the path has to enter $G_{\ell_1+1}$ immediately by Assumption \ref{ass:1}~b'): then there is $\vep>0$ such that $X(t)\in G_{\ell_1+1}$ for $0< t\le 2\vep$, and we argue as above in restriction to the compact $[\vep,2t_0]$.

7) Fix $x\in C_n$ and consider $m > n$. Viewing  $\vph_m^{({\tt h},x)}$ of (\ref{controlsystem-n}) as controls for $X_m$ on the time interval $[0,2t_0]$, there is a sequence $({\tt h}_m(W))_m$ of adapted linear interpolations of the driving Brownian path $(W_t)_{t\ge 0}$ with the property 
$$
\mbox{for every $m>n$} \;:\quad 
E_x \left( \, \sup_{0\le t\le 2t_0}\left| X_m(t) - \vph_m^{({\tt h}_m(W),x)}(t)\right| ^2 \,\right)  \;\;\le\;\;  2^{-3m} \;.
$$
To see this, recall that all equations (\ref{sde-n}), $n\ge 1$, are driven by the same $W$, and construct $({\tt h}_m(W))_m$ by extracting a subsequence from the sequence of adapted linear interpolations in (2.6)+(2.7) of \cite{MilSan-94}. Transform the last assertion into 
$$
P_x\left( \, \sup_{0\le t\le 2t_0}\left| X_m(t) - \vph_m^{({\tt h}_m(W),x)}(t)\right| > 2^{-m} \,\right)  \;\;\le\;\;  2^{-m}  
\quad\mbox{for all $m>n$} \;.  
$$
Since $X$ coincides with $X_m$ up to time $T_{m+1}$, Borel-Cantelli then yields 
$$
\sup_{0\le t\le 2t_0}\left| X(t) - \vph_m^{({\tt h}_m(W),x)}(t)\right| \;\le\;  2^{-m} \quad\mbox{for eventually all $m$}  
$$
$P_x$-almost surely. 
Combining this last assertion with the result of step 6) we see that the event  
\beqq\label{property-1}
\left\{\begin{array}{l}
\mbox{there exists some $\wt m<\infty$ such that $\,X_t \in G_{\wt m + 1}\,$ for  $0< t\le 2 t_0$ } \;,\\ 
\mbox{and}\;\;  \sup\limits_{0\le t\le 2t_0}\left| X(t) - \vph_m^{({\tt h}_m(W),x)}(t)\right| \;\le\;  2^{-m} \;\;\mbox{for all $m> \wt m$} 
\end{array}\right.
\eeqq
has $P_x$-probability equal to $1$.

8) To conclude the proof, fix $\om$ belonging to the event (\ref{property-1}). Note that  $X(t) \in G_{\wt m +1}={\rm int}(C_{\wt m +1})$ for $0<t\le 2t_0$ implies $2t_0 < T_{\wt m +1}$. Write for short $\,f_m := \vph_m^{({\tt h}_m(W),x)}\,$ in (\ref{property-1}). By step 7), $f_m$ is a control path for $X$ up to time $T_{m+1}$, and leaves $C_{m+1}$ at time $t^m_{m+1}({\tt h}_m(W),x)$. All this depends on $\om$. By definition in (\ref{pastingtogether}), we can embed $\,f_m\,1_{[\, 0 \,,\, t^m_{m+1}({\tt h}_m(W),x)\, ]}\,$ into $\,\vph^{({\tt h}_m(W),x)}\,$ which admits some (finite or infinite) life time ${\tt s}({\tt h}_m(W),x)$. Steps 4) and 5) then show  
\beqq\label{order_of_times}
{\tt s}({\tt h}_m(W),x) \;\;\ge \;\; t^m_{m+1}({\tt h}_m(W),x) \;>\;  t^{\wt m}_{\wt m+1}({\tt h}_m(W),x)
\eeqq
if $t^{\wt m}_{\wt m+1}({\tt h}_m(W),x)<\infty$, for every $m> \wt m$. Keeping $\om$ fixed, we have functions taking values in $C_{\wt m +1}$  
$$
t \;\lra\;  f_m(t) = \vph^{({\tt h}_m(W),x)}(t) \;\;,\;\;  0\le t\le t^{\wt m}_{\wt m+1}({\tt h}_m(W),x)  
$$
for $m> \wt m$, and a function 
$$
t \;\lra\; X(t) \;\;,\;\; 0\le t\le T_{\wt m +1}
$$
which for $0< t\le 2t_0 < T_{\wt m +1}$ remains in the interior $G_{\wt m +1}={\rm int}(C_{\wt m +1})$ by (\ref{property-1}). From uniform convergence on $[0,2t_0]$ as $m\to\infty$ in (\ref{property-1}) we obtain 
$$
2t_0 \;\;\le \;\; \liminf\limits_{m}\; t^{\wt m}_{\wt m+1}({\tt h}_m(W),x)  
$$
and thus in virtue of  (\ref{order_of_times})
$$
2t_0 \;\; \le \;\; \liminf\limits_{m}\; {\tt s}({\tt h}_m(W),x) \;. 
$$
This establishes  that  the support of $Q^{t_0}_x = \call(\, X\, 1_{[0,t_0]} \mid P_x)$ in the sense of uniform convergence in $C([0,t_0],E)$ is contained in the closure of 
$$
\left\{  \left( \vph^{({\tt h}, x)}(t) \right)_{0\le t\le t_0} : \,{\tt h}\in{\tt H}\;\;\mbox{such that $s({\tt h}, x)>t_0$}\; \right\} 
$$
and thus in the closure of (\ref{newcontrolsystems}). The proof of Theorem \ref{theo:4bis} b) is finished. {\hfill $\bullet$ \vspace{0.25cm}}\\

{\bf Proof of Theorem \ref{theo:4bis} a) } 
The inverse inclusion now follows easily. Suppose that $x \in C_n.$ Since $ \vph^{({\tt h }, x)} (s) \in {\rm int} ( E) $ for all $0 < s \le \wt T$ where $\wt T > t_0$, there exists necessarily $ m \geq n $ such that $  \vph^{({\tt h }, x)} (s)  \in G_{m+1} $ for all $ s \in ] 0, \wt T ].$ Applying step 5) of the first part of the proof, we know that $ \vph^{({\tt h }, x)}  $ coincides with $\vph_m^{({\tt h }, x)} $ on $ [ 0, t^m_{m+1} ( {\tt h }, x ) ] . $ Since 
$ t^m_{m+1} ( {\tt h }, x ) \geq \wt T > t_0,$  
this implies that $ \vph^{({\tt h }, x)} =  \vph_m^{({\tt h }, x)} $ on $[0,\wt T].$ Applying step 2) of the first part of the proof, we obtain 
$$
\;{\vph_m^{ ( {\tt h }, x )}}_{| [0, \wt T]} \in \overline{ {\rm supp} \left( Q_x^{m, \wt T } \right)}\; ,
$$
where $ Q_x^{m, \wt T} $ denotes the law of the solution $ (X^m_t)_{0 \le t \le \wt T} $ starting from $X_0 = x .$ We also clearly have that $(X^m_t)_{0 \le t \le \wt T} $ stays in $G_{m+1} $ with positive $P_x$-probability. Since $ X^m_t = X_t $ for all $t \le T_{m+1}, $ this implies that $T_{m+1}\ge \wt T$ with positive $P_x$-probability, thus 
$ 
\;\vph^{ ( {\tt h }, x )}_{| [0, \wt T]}\in \overline{ {\rm supp} \left( Q_x^{\wt T } \right)}\;.
$
%implying the assertion. 
{\hfill $\bullet$ \vspace{0.25cm}}

%%%%%%%%%%%%%%%%%%%%%%%%%%%%%%
%%%%%%%%%%%%%%%%%%%%%%%%%%%%%%

%%%%%%%%%%%%%%%%%%%%%e
\end{document}